\title{\bf Multigraph limits, unbounded kernels,\\ and Banach space
decorated graphs}
\author{\sc D\'avid Kunszenti-Kov\'acs\footnote{Research supported by  ERC
Advanced Research Grant No.~227701, the Bolyai Research Grant of the Hungarian Academy of Sciences and ERC Consolidator Grant No.~648017.}
\\
\rm Alfr\'ed R\'enyi Institute of Mathematics, Budapest, Hungary\\
\sc L\'aszl\'o Lov\'asz\footnote{Research supported by  ERC
Advanced Research Grant No.~227701 and ERC Synergy Grant  No. 810115.}
\\
\rm Alfr\'ed R\'enyi Institute of Mathematics\\ Institute of Mathematics, E\"otv\"os Lor\'and University,
Budapest, Hungary\\
\sc Bal\'azs Szegedy\footnote{Research supported by ERC Consolidator Grant No.~617747 and the Hungarian National Excellence Grant 2018-1.2.1-NKP-00008.}\\
\rm Alfr\'ed R\'enyi Institute of Mathematics, Budapest,
Hungary
\\[1cm]}

\documentclass[11pt,a4paper]{article}

\usepackage{hyperref}
\usepackage{enumerate,xcolor,tensor,bbm}
\usepackage{amsmath,amsthm,amssymb}
\usepackage{mathrsfs}
\usepackage{srcltx}

\newcommand{\mk}{\mathfrak}
\newcommand{\mc}{\mathcal}
\newcommand{\ms}{\mathscr}

\newcommand{\mf}{\mathbf}
\newcommand{\mb}{\mathbb}
\newcommand{\mr}{\mathrm}

\def\BB{\mathcal{B}}\def\CC{\mathcal{C}}

\def\JJ{\mathcal{J}}\def\KK{\mathcal{K}}\def\LL{\mathcal{L}}
\def\MM{\mathcal{M}}
\def\PP{\mathcal{P}}\def\RR{\mathcal{R}}

\def\WW{\mathcal{W}}\def\XX{\mathcal{X}}
\def\ZZ{\mathcal{Z}}

\def\Fs{\mathscr{F}}
\def\Gs{\mathscr{G}}

\def\Ps{\mathscr{P}}\def\Rs{\mathscr{R}}

\def\Bb{\mathbf{B}}
\def\Fb{\mathbf{F}}
\def\Gb{\mathbf{G}}\def\Hb{\mathbf{H}}

\def\Gbb{\mathbb{G}}

\def\Rf{\mathfrak{R}}

\def\phib{{\boldsymbol\varphi}}

\long\def\ignore#1{}

\def\rand{\Rf}
\def\pcut{{{}_\boxtimes}}
\def\wt{\widetilde}
\def\wh{\widehat}
\def\inj{{\rm inj}}

\def\one{{\mathbbm1}}
\def\eps{\varepsilon}
\def \ph {\varphi}

\def \Q {\mathbb{Q}}
\def \R {\mathbb{R}}

\def \N {\mathbb{N}}

\def \e {\varepsilon}
\def\E{{\sf E}}

\def \lin {\operatorname{lin}}

\theoremstyle{plain}
\newtheorem*{thm*}{Theorem}
\newtheorem{thm}{Theorem}[section]
\newtheorem{lemma}[thm]{Lemma}
\newtheorem{prop}[thm]{Proposition}
\newtheorem{cor}[thm]{Corollary}

\theoremstyle{definition}
\newtheorem{definition}[thm]{Definition}
\newtheorem{example}[thm]{Example}

\newtheorem{remark}[thm]{Remark}
\newtheorem{ass}[thm]{}

\newenvironment{proof*}[1]{\medskip\noindent{\bf Proof of #1.}}{\hfill$\square$\medskip}
\begin{document}		

\maketitle

\begin{abstract}
We present a construction that allows us to define a limit object of
Banach space decorated graph sequences in a generalized homomorphism
density sense. This general functional analytic framework provides a universal language for various combinatorial
 limit notions. In particular it makes it possible to assign limit objects to
multigraph sequences that are convergent in the sense of node-and-edge
homomorphism numbers, and it generalizes the limit theory for graph
sequences with compact decorations.
\end{abstract}

\tableofcontents

\section{Introduction}

The motivation for this paper was to provide a framework for a theory
of convergence and limits of graphs with unbounded edge
multiplicities, along the lines of the limit theory for dense simple
graphs developed by Borgs, Chayes, Lov\'asz, S\'os and Vesztergombi
\cite{BCLSV1,BCLSV2} and Lov\'asz and Szegedy \cite{LSz1}. Key
elements of this theory are the notions of cut distance and subgraph
densities, the definitions of convergence and limit objects, the
Regularity Lemma (in its weak form due to Frieze and Kannan
\cite{FK}), along with the Counting Lemma.

In the paper \cite{LSz8} (posted on the Arxiv, but not published; see
also \cite{Hombook}, Section 17.1), the second and third authors worked out a
theory of convergence and limits of simple graphs whose edges are
decorated by points from some compact space. (Ordinary simple graphs
can be considered as complete graphs with edges decorated by elements
of a two-point space.) One could say that the theory of undecorated
simple graphs extends to this case in a rather straightforward manner
(at least as soon as the appropriate formulations are found).
Edge-weighted graphs and multigraphs fit in this framework, provided
the edge weights/multiplicities are bounded (but for unbounded
multiplicities or edge weights one has to do more, as we shall see).
We note that this paper took the approach of defining convergence via weak convergence of the distribution of samples, and this gives a link to exchangeability and Aldous' representation theorem. It turns out that whilst sampling convergence is equivalent to homomorphism convergence notions
for compact decorations, this does not longer hold if the compactness condition is waived (see Section \ref{sect:exch}).

A limit theory for convergence of multigraphs was worked out by
Kolossv\'ary and R\'ath \cite{KoRath}, and essentially the same
results can be derived from the limit theory of compact decorated
graphs using the one-point compactification of the set of integers to
encode the edge multiplicities. The limit objects can be described by
functions on $[0,1]^2$ whose values are probability distributions on
nonnegative integers.

Let us describe in a few words the general framework for graph
convergence theories. We start with defining the number of
occurrences of a ``small'' graph $F$ in a ``big'' graph $G$. In the
case of simple graphs, one can use the number of homomorphisms
(adjacency-preserving maps) $\hom(F,G)$ from $F$ to $G$. One also
needs the (normalized) homomorphism density
\begin{equation}\label{EQ:HOM-DENS}
t(F,G)=\frac{\hom(F,G)}{|V(G)|^{|V(F)|}}.
\end{equation}

A key notion in these theories is that of convergence of a graph
sequence, which is defined by specifying an appropriate family of
{\it test graphs}, and then saying that a sequence of graphs
$(G_1,G_2,\dots)$ is {\it convergent}, if $t(F,G_n)$ is convergent
for every test graph $F$. In the theory of convergence of simple
graphs, the family of simple graphs is the right (in a sense, only
reasonable) choice for test graphs. The limiting values of these
densities can be represented by limit objects called {\it graphons},
which in the case of simple graphs are symmetric measurable functions
$[0,1]^2\to[0,1]$.

The motivation of this paper is to work out a limit theory for
convergence of multigraphs. Whether or not the results of
Kolossv\'ary and R\'ath \cite{KoRath} can be viewed as a solution of
the problem of multigraph convergence depends on how we define
homomorphisms between two multigraphs $F$ and $G$.

One natural definition is that of {\it node-and-edge homomorphism}:
this is a pair of maps $\ph:~V(F)\to V(G)$ and $\psi:~E(F)\to E(G)$
such that if $e\in E(G)$ connects $i$ and $j$, then $\psi(e)$
connects $\ph(i)$ and $\ph(j)$. A different definition is that of a
{\it node-homomorphism}: a map $V(F)\to V(G)$ such that the
multiplicity of the image of an edge is not less than the
multiplicity of the edge. If both $F$ and $G$ consist of two nodes
connected by two edges, then the number of node-homomorphisms $F\to
G$ is $2$, while the number of node-and-edge homomorphisms is $8$.

In this paper, we consider node-and-edge homomorphisms, and for two
multigraphs, we denote by $\hom(F,G)$ their number. We define
homomorphism densities and convergence based on this definition. The
results of Kolossv\'ary and R\'ath are based on node-homomorphisms.
It turns out that these two notions of convergence are not equivalent
(see Section \ref{SEC:EXAMPLES}, and also \cite{Hombook}, Chapter
17).

In fact, we consider a more general model, namely a limit theory of
graphs whose edges are decorated by elements from a Banach space, and
where the test graphs are decorated from the pre-dual space. This
will include the convergence theory of compact decorated graphs as
well. Along the way, we show that with a modified notion of cut
distance (which we call ``jumble distance'') one can state a prove an
appropriate Weak Regularity Lemma and a Counting Lemma.

Recently Borgs, Chayes, Cohn and Zhao \cite{BCCZ} developed a theory
for $L^p$-graphons (unbounded symmetric functions in the space
$L^p([0,1]^2)$ for some $2<p<\infty$), and graph sequences convergent
to them. They prove appropriate versions of the Regularity and
Counting Lemmas. Not every graph has a finite density in such a
kernel, and accordingly, they limit the set of test-graphs to simple
graphs with degrees bounded by $p$. Their set-up is more general than
ours in the sense that we work with a more restricted family of
unbounded kernels, namely kernels in $\LL=\bigcap_{1\leq p<\infty}
L^p_{\mr{sym}}([0,1]^2)$. On the other hand, we allow arbitrary multigraphs as
test graphs (more generally, decorating by elements of a Banach
space). So the two theories don't seem to contain each other (but
perhaps a common generalization is possible).

Using random graphs generated by Banach space valued graphons, we show that every element of the space of limit objects that we define arises as a limit of a convergent sequence of decorated graphs, and that this space is closed under our convergence notion.

Although we do not in this paper investigate the question of uniqueness of Banach space graphons (we refer to \cite{DKK} by the first author for details on that subject), we remark here that Examples \ref{EXA:TWO} and \ref{ex:counter} do show that indeterminacies in the Stieltjes/Hamburger moment problems are an extra natural obstacle to uniqueness for unbounded graphons, beyond the usual weak isomorphism equivalence (see, e.g., \cite[Sections 7.3 and 10.7]{Hombook}).

\section{Decorated graphs and graphons}

\subsection{Decorated graphs and graphons}

If $\XX$ is any set, an {\it $\XX$-decorated graph} is a graph where
every edge $ij$ is decorated by an element $X_{ij}\in\XX$. An
$\XX$-decorated graph will be denoted by $(G,g)$, where $G$ is a
simple graph (possibly with loops), and $g:~E(G)\to\XX$. We will see
several examples in Section \ref{SEC:EXAMPLES} how decorations can be
used to express weights, multiple edges, and more.

In our setup, we will consider decorations by elements of Banach
spaces. Let $\BB$ be a separable Banach space, let $\ZZ$ denote its dual. The elements of $\BB$ act on $\ZZ$ as
bounded linear functionals in the canonical way, and vice versa; the
action of $b\in\BB$ on $z\in\ZZ$ will be denoted by $\langle
b,z\rangle$. We will use ``small'' $\BB$-decorated graphs to probe
``large'' $\ZZ$-decorated graphs.

Let $(G,g)$ be an $X$-decorated graph, where $G$ is a graph with $m$
edges, and $X$ is a Banach space. We define
\[
\|g\|_p = \Bigl(\frac{1}{m}\sum_{e\in E(G)}
\|g(e)\|_X^p\Bigr)^{1/p}, \quad
\|g\|_\infty = \max_{e\in E(G)}
\|g(e)\|_X, \quad\text{and}\quad
\Pi_g=\prod_{e\in E(G)}\|g(e)\|_X.
\]
Clearly $\Pi_g \le \|g\|_1^m$.

To define ``decorated graphons'' we need to become more technical. We
set $\LL=\bigcap_{1\leq p<\infty} L^p_{\mr{sym}}([0,1]^2)$.
For $b\in\BB$ and $W:~[0,1]^2\to\ZZ$, let the function $\langle
b,W\rangle:~[0,1]^2\to\R$ be defined by
\[
\langle b, W \rangle(x,y):=\langle b,W(x,y) \rangle.
\]
\begin{definition}
A function $W:~[0,1]^2\to\ZZ$ is called \emph{weak-* measurable} if for any $b\in\BB$, the function $\langle b,W\rangle$ is measurable.
\end{definition}

\begin{definition}
A symmetric weak-* measurable function $W:~[0,1]^2\to\ZZ$ is called
a $\ZZ$-\emph{graphon} if the function $(x,y)\mapsto \|W(x,y)\|_\ZZ$
lies in $\LL$. Note that this function is measurable, since $\BB$ is separable, and for a countable dense subset $\Fs\subset\BB$ we have
\[
\|W(x,y)\|_\ZZ=\sup_{f\in\Fs\backslash\{0\}}\frac{|\langle f,W(x,y)\rangle|}{\|f\|_\BB}.
\]
Let the space of $\ZZ$-graphons be denoted by
$\WW_\ZZ$. We set
\[
\|W\|_p := \bigl\|\|W(.,.)\|_\ZZ\bigr\|_p.
\]
(i.e., we take the $\ZZ$-norm of $W(x,y)$ for every $x,y\in[0,1]$,
and then take the $L^p$-norm of the resulting function).
\end{definition}

Also, if $W$ is a $\ZZ$-graphon, then $\langle b,W\rangle\in\LL$.
Indeed,
\begin{equation}\label{EQ:PW-L}
\|\langle b,W\rangle\|_p \le \bigl\|\|b\|_\BB \|W(.,.)\|_\ZZ\bigr\|_p
=\|b\|_\BB\|W\|_p <\infty.
\end{equation}

Note that the above measurability notion is not one usually encountered when looking at Banach space valued functions, for the simple reason that generally a predual $\mc{B}$ may not exist to the range space $\mc{Z}$ of the functions at hand. The two usual notions are weak and strong/Bochner measurability, and we shall briefly hint at why weak-* measurability is instead the correct notion for our purposes.\\
Recall that a function $W:~[0,1]^2\to\ZZ$ is called \emph{weakly measurable} if for any $b\in\ZZ'$, the function $\langle b,W\rangle$ is measurable. Note that unless $\BB$ is reflexive (which in our applications it typically will not be), weak-* measurability is strictly weaker due to the natural identification of $\BB$ with a subspace of $\ZZ'$.\\
A function $W:~[0,1]^2\to\ZZ$ is called {\it Bochner measurable}
if it equals to the limit of a sequence of measurable functions with
countable range almost everywhere. This in turn is by the Pettis measurability theorem equivalent to being weakly measurable and essentially separably valued (i.e., one may obtain a separable range by deleting a nullset from the domain). Note that whenever $\ZZ$ is separable, Bochner and weak measurability are equivalent.

As will become clear in the next section, the homomorphism densities which will be at the core of our convergence notion can be defined just as well using $\ZZ'$ instead of $\BB$. Also, our aim is to define a space of graphons that acts as a closure of the space of decorated graphs, and so is as small as possible. This indicates that we should opt for the strongest possible measurability restrictions, favouring strong or weak measurability. However, (norm) closed balls in $\ZZ$ are not compact with respect to the weak topology, unless we have reflexivity and $\BB=\ZZ'$ anyway. Therefore, if we wish to use $\ZZ$-decorated graphs and test with elements from $\ZZ'$, we might after the limit transition end up with something $\ZZ''$-valued, which again brings us to a range with a predual. In addition, there would also be some -- related -- issues defining the stepping operators (Section \ref{sect:weak_reg}), as weak/Pettis integrability does not follow from weak measurability even for bounded functions. For more details on the subtleties of measurability and integrability of Banach space valued functions, we refer to \cite{Musial}.

\subsection{Homomorphism densities}

For every $\BB$-decorated simple loopless graph $\Fb=(F,f)$ and
$\ZZ$-decorated completely looped complete graph $\Gb=(G,g)$, we define
\[
\hom(\Fb,\Gb)=:
\sum_{\ph:V(F)\to V(G)} \prod_{e\in E(F)} \langle f(e), g(\ph(e))\rangle.
\]
Loops in $\Gb$ are unavoidable if we want to allow non-injective vertex maps, whereas not having loops on $\Fb$ conforms with the classical theory, and avoids having to cope with an extra component for the limit object (a ``diagonal'').
The homomorphism density $t(\Fb,\Gb)$ is defined by \eqref{EQ:HOM-DENS}.
Note that we can write
\begin{equation}\label{EQ:HOM-DENS-E}
t(\Fb,\Gb)=\E\Bigl(\prod_{e\in E(F)} \langle f(e), g(\ph(e))\rangle\Bigr),
\end{equation}
where the expectation is taken over uniform random maps $\ph:~V(F)\to
V(G)$.

In the definition of the homomorphism number between finite graphs,
it is sometimes convenient to restrict the summation to injective
mappings. For every $\BB$-decorated graph $\Fb=(F,f)$ with $k$ nodes
and $\ZZ$-decorated graph $\Gb=(G,g)$ with $n$ nodes, we define
\[
\inj(\Fb,\Gb)=:
\sum_{\ph:V(F)\hookrightarrow V(G)}
\prod_{e\in E(F)} \langle f(e), g(\ph(e))\rangle,
\]
and
\[
t_\inj(\Fb,\Gb) = \frac{\inj(\Fb,\Gb)}{n(n-1)\dots(n-k+1)}.
\]

We need some elementary estimates, summarized in the following lemma.

\begin{lemma}\label{LEM:T-INJ}
Let $\Fb=(F,f)$ be a $\BB$-decorated graph with $k$ nodes and $l$
edges, and let $\Gb=(G,g)$ be a $\ZZ$-decorated graph with $n$ nodes.
Then
\[
|t(\Fb,\Gb)|\le \Pi_f \|g\|_l^l,\qquad |t_\inj(\Fb,\Gb)|\le \frac{n}{n-1} \Pi_f \|g\|_l^l (\mbox{ for } n\geq 2),
\]
and assuming $n\geq2$ and $g(u,u)=0\in\ZZ$ for all $u\in V(G)$,
\[
|t_\inj(\Fb,\Gb)-t(\Fb,\Gb)| \le \frac{k(k-1)}{n-1} \Pi_f \|g\|_l^l.
\]
\end{lemma}

\begin{proof}
Let $\phib$ be a random map $V(F)\to V(G)$, then
\begin{align*}
|t(\Fb,\Gb)|&\le \E\Bigl(\prod_{e\in E(F)} |\langle f(e),
g(\phib(e))\rangle|\Bigr)
\le\E\Bigl(\prod_{e\in E(F)} \|f(e)\|_\BB
\|g(\phib(e)\|_\ZZ\Bigr)\\
&=\Pi_f \E\Bigl( \prod_{e\in E(F)}
\|g(\phib(e))\|_\ZZ\Bigr).
\end{align*}
Here $\phib(e)$ is uniform over all pairs in $V(G)\times V(G)$. Using
H\"older's Inequality,
\begin{align*}
&|t(\Fb,\Gb)|\le \Pi_f  \Bigl(\prod_{e\in E(F)}
\E \bigl(\|g(\phib(e)\|_\ZZ^l\bigr) \Bigr)^{1/l} = \Pi_f \|g\|_l^l.
\end{align*}
The second inequality follows by analogue computation, with the only modification being that working with injective maps, $\phib(e)$ will be uniform on the off-diagonal elements of $V(G)\times V(G)$, hence the additional $n^2/n(n-1)=n/(n-1)$ factor.\\
For the third inequality, we shall make use of the fact that injective maps are counted both in $\hom(\Fb,\Gb)$ and $\inj(\Fb,\Gb)$, so a number of terms will cancel despite the different normalization. Also, by the Bernoulli inequality
\[
(n)_k\geq n^k\left(1-\sum_{j=0}^{k-1}\frac{j}{n}\right)=n^{k-1}\left(n-\frac{k(k-1)}{2}\right).
\]
Therefore we have
\begin{align*}
|t(\Fb,\Gb)-t_\inj(\Fb,\Gb)|&=\frac{1}{n^k}\left|\hom(\Fb,\Gb)-\frac{n^k}{(n)_k}\inj(\Fb,\Gb)\right|\\
&=\frac{1}{n^k}\left|\left(\sum_{\substack{\ph:V(F)\to V(G)\\\ph\mr{\; not\; injective }}} \prod_{e\in E(F)} \langle f(e), g(\ph(e))\rangle\right)-\frac{n^k-(n)_k}{(n)_k}\inj(\Fb,\Gb)\right|\\
&\leq \frac{1}{n^k}\left|\left(\sum_{\substack{\ph:V(F)\to V(G)\\\ph\mr{\; not\; injective }}} \prod_{e\in E(F)} \langle f(e), g(\ph(e))\rangle\right)\right|+
\frac{n^k-(n)_k}{n^k}\left|t_\inj(\Fb,\Gb)\right|\\
&\leq \frac{1}{n^k}\left|\left(\sum_{\substack{\ph:V(F)\to V(G)\\\ph\mr{\; not\; injective }}} \prod_{e\in E(F)} \langle f(e), g(\ph(e))\rangle\right)\right|+
\frac{k(k-1)}{2(n-1)}\Pi_f \|g\|_l^l
\end{align*}
For the first term, an argument identical to the previous leads to
\[
\frac{1}{n^k}\left|\left(\sum_{\substack{\ph:V(F)\to V(G)\\\ph\mr{\; not\; injective }}} \prod_{e\in E(F)} \langle f(e), g(\ph(e))\rangle\right)\right|
\leq \frac{n^k-(n)_k}{n^k}\Pi_f  \Bigl(\prod_{e\in E(F)}
\E_{\mr{ninj}} \bigl(\|g(\phib(e)\|_\ZZ^l\bigr) \Bigr)^{1/l},
\]
where the expectation is taken over all non-injective maps. It remains to investigate the distribution of $\phib(e)$. Note that since the distribution under all maps would be uniform on $V(G)\times V(G)$, whereas under injective maps uniform on the off-diagonal elements, we see that for non-injective maps, we have $\phib(e)=(u,v)\in V(G)\times V(G)$ with probability
\begin{align*}
\frac{n^k-(n)_k\frac{n}{n-1}}{n^k-(n)_k} \cdot\frac{1}{n^2}\qquad \mr{if}\; u\neq v;
\\
\frac{n^k}{n^k-(n)_k}\cdot \frac{1}{n^2} \qquad \mr{if}\; u=v.
\end{align*}
However $g(u,u)=0\in\BB$ for all $u\in V(G)$,
meaning
\[
\E_{\mr{ninj}} \bigl(\|g(\phib(e)\|_\ZZ^l\bigr)\leq \frac{n^k-(n)_k\frac{n}{n-1}}{n^k-(n)_k} \|g\|_l^l<  \|g\|_l^l,\]
and so
\[
|t(\Fb,\Gb)-t_\inj(\Fb,\Gb)|\leq 
\frac{k(k-1)}{2n}\Pi_f \|g\|_l^l
+\frac{k(k-1)}{2(n-1)}\Pi_f \|g\|_l^l,
\]
and the conclusion follows.
\end{proof}

Another type of decoration we consider is $\LL$-decoration, given by
a map $w:~E(F)\to\LL$ (so every edge $ij$ is decorated by a function
$w_{ij}\in\LL$). In this case, a ``homomorphism density'' can be
defined for a single graph: for an $\LL$-decorated graph $(F,w)$ on
$V(F)=[k]$, we define
\[
t(F,w):=\int\limits_{[0,1]^{V(F)}} \prod_{ij\in E(F)}
w_{ij}(x_i,x_j) dx_1\ldots dx_k.
\]
This integral is well defined: if $m=|E(F)|$, then by H\"older's
Inequality,
\begin{equation}\label{EQ:HOLD}
\int\limits_{[0,1]^{V(F)}} \Bigl|\prod_{ij\in E(F)}
w_{ij}(x_i,x_j)\Bigr| dx_1\ldots dx_k \le  \prod_{ij\in E(F)} \|w_{ij}\|_m,
\end{equation}
which is finite.

Most of the time we need the following special case. Consider a
$\BB$-decorated graph $\Fb=(F,f)$ on $[k]$ and a $\ZZ$-graphon $W$.
Then $w_{ij}=\langle f_{ij}, W\rangle$ defines an $\LL$-decoration of
$F$, and the previous definition specializes to
\[
t(\Fb,W):=t(F,w)=\int\limits_{[0,1]^{V(F)}} \prod_{ij\in E(F)}
\langle f(ij),W(x_i,x_j)\rangle dx_1\ldots dx_k.
\]

Finally, as a variation on the above, but without the dependence of the second variable on what simple graph $F$ we chose. For any $\ms{F}$-decorated graph $\Fb$ and $s\in \mc{L}^\ms{F}$ we may define
\[
t(\Fb,s):=\int\limits_{[0,1]^{V(F)}} \prod_{ij\in E(F)}
s_{f_{ij}}(x_i,x_j) dx_1\ldots dx_k.
\]

We can think of getting information about a $\ZZ$-graphon $W$ by
``probing'' it with various ``small'' $\BB$-decorated graphs $\Fb$.
It is often natural to restrict the decoration of our test graphs to
a subset $\Fs\subseteq\BB$, for which homomorphism numbers of
$\Fs$-decorated graphs carry special combinatorial information. The
family $\Fs$ will be countable in our examples. We usually assume
that the set $\Fs$ is {\it generating} in $\BB$ (meaning that
$\overline{\lin \Fs}=\BB$ -- this is sometimes also called total or fundamental). In this case, the values $t(\Fb,W)$ carry
the same information about $W$ if we restrict $f$ to
$\Fs$-decorations than if we allow all $\BB$-decorations. If $\BB$ is
finite dimensional, then any basis can be chosen for $\Fs$, but the
choice of the basis does not actually matter, as illustrated by the examples in
the next section.

\subsection{Examples}\label{SEC:EXAMPLES}

A large variety of examples comes from {\it compact decorated
graphs}, i.e., $\KK$-decorated graphs where $\KK$ is a compact
Hausdorff space. To capture convergence of $\KK$-decorated graphs, we
need to consider $\CC(\KK)$-decorated graphs (where $\BB=\CC(\KK)$ is
the space of continuous real functions on $\KK$, with the supremum
norm). The dual space $\ZZ=\RR(\KK)$  is the set of Radon measures on
the Borel sets of $\KK$. This contains probability measures
concentrated on a single point, and thus $\KK$-decorated graphs can
be thought of as special $\ZZ$-decorated graphs. We are also
interested in selecting a ``nice'' countable generating set
$\Fs\subseteq \BB$. It is interesting to note that different choices
of $\Fs$ carry different combinatorial information.

Many examples of $\KK$-decorated graphs with combinatorial
significance were discussed in \cite{LSz8} and also in
\cite{Hombook}, Chapter 17, and we only mention them briefly.

\begin{example}[Simple graphs]\label{EXA:SGRAPH}
Let $\KK=\{0,1\}$ be the discrete space with two elements
corresponding to ``non-edge'' and ``edge''. The space $\BB$ consists
of all maps $\{0,1\}\to\R$, i.e., of all pairs $(f(0),f(1))$ of real
numbers. Clearly, the dual space $\ZZ$ can also be identified with
$\R^2$.

A natural generating subset $\Fs$ consists of the pairs $(1,1)$ and
$(0,1)$. Homomorphism density corresponds to that for simple graphs.
Every probability distribution on $\KK$ can be represented by a
number between $0$ and $1$, which is the probability of the element
``edge''. So every symmetric measurable function $W:~[0,1]^2\mapsto
[0,1]$ defines a $\ZZ$-graphon.

One may, however, take another basis in $\BB$, namely the pairs
$(0,1)$ and $(1,0)$. Then again $\Fs$-decorated graphs can be thought
of as simple graphs, and $\hom(F,G)$ counts the number of maps that
preserve both adjacency and non-adjacency.
\end{example}

\begin{example}[Bounded multigraphs and multi-test-graphs]
\label{EXA:MULTGRAPH}
Let $G$ be a multigraph with edge multiplicities at most $d$. Then
$G$ can be thought of as a $\KK$-decorated graph, where
$\KK=\{0,1,\dots,d\}$. This can be modeled with
$\BB\cong\ZZ\cong\R^{d+1}$.

An interesting basis in $\BB$ consists of the functions
$\Fs=\{1,x,\dots,x^d\}$. We can represent an $\Fs$-decorated graph by
a multigraph with edge multiplicities at most $d$, where an edge
decorated by $x^i$ is represented by $i$ parallel edges. The
advantage of this is that $\hom(F,G)$ is then the number of
node-and-edge homomorphisms of $F$ into $G$ as multigraphs, and so
$t(F,G)$ is the node-and-edge homomorphism density.

Taking the standard basis $\{e_0,e_1,\dots,e_d\}$ in $\BB$, we can
represent an edge label $e_i$ by $i$ parallel edges. In this case,
$\hom(F,G)$ counts multiplicity-preserving homomorphisms of $F$ into
$G$.

As a third possibility, we can consider the basis vectors
$f_i=e_0+\dots+e_i$ ($i=0,\dots,d$). Representing an edge label $f_i$
by $i$ parallel edges, $\hom(F,G)$ counts node-homomorphisms of $F$
into $G$.

This shows that in the case of bounded edge-multiplicities, different
ways of counting homomorphisms between multigraphs are essentially
equivalent, they differ in a simple basis transformation in the space
$\BB$.
\end{example}

\begin{example}[Bounded weighted graphs and multi-test-graphs]
\label{EXA:WGRAPH} This example is a rather straightforward extension
of the previous one. Let $\KK\subseteq\R$ be a bounded closed
interval. Let $\Fs$ be the collection of monomial functions $x\mapsto
x^j$ for $j\in\N$ on $\KK$ (we denote by $\N$ the set of nonnegative
integers, and by $\N^*$, the set of positive integers). The linear
hull of $\Fs$ is dense in $\BB=\CC(\KK)$. It is natural to consider
an $\Fs$-decorated graph $F$ as a multigraph, and then $t(F,G)$ is
the weighted homomorphism number as defined e.g.\ in \cite{Hombook}.
\end{example}

Compact decorations do not utilize the full strength of the general
theory developed below. We discuss a couple of examples of this kind,
and will return to these examples in Section \ref{section:multi}.

\begin{example}[Multigraphs and simple test-graphs]
\label{EXA:MULTGRAPH2} Let us consider multigraphs with unbounded
edge-multiplicity, and simple graphs as test-graphs. It turns out
that in this case, multigraphs can be thought of as $\N$-decorated
simple graphs, and the fact that the edgeweights are nonnegative
integers plays no role; so we can take $\ZZ=\BB=\R$, and consider the
basis $\Fs=\{1\}$ in $\BB$. If $F$ is a simple ($\Fs$-decorated)
graph and $G$ is an edge-weighted (complete) graph, then $\hom(F,G)$
is the homomorphism number into $G$ as a multigraph.
\end{example}

\begin{example}[Multigraphs and multi-test-graphs]
\label{EXA:MULTGRAPH3} Consider multigraphs with unbounded
edge-multiplicity, and multigraphs as test-graphs. We have already
seen that homomorphisms of a multigraph into another can be defined
in different ways. In the bounded case, these notions turned out to
be essentially the same, but in the unbounded case, the
correspondence is more subtle.

We sketch the idea how to fit convergence according to node-and-edge
homomorphism densities into our framework. Let $\BB=\R[X]$ be the
space of polynomials in one variable, and let $\ZZ$ be the space of real
sequences with finite support. For $f\in\BB$ and
$a=(a_1,a_2,\dots)\in\ZZ$, let us define
\[
\langle f,a\rangle = \sum_{i=0}^\infty a_if(i).
\]
We encode a multigraph $F$ by decorating each edge $e\in E(F)$ with
multiplicity $m$ by the polynomial $X^m$, to get an edge-decorated
simple graph $\wh{F}$. We encode a ``target'' multigraph $G$ by
labeling each edge $e\in E(F)$ with multiplicity $m$ by sequence
$e_m$ with a single $1$ in the $m$-th position, to get an
edge-decorated complete graph $\wt{G}$. Then $\hom(\wh{F},\wt{G})
=\hom(F,G)$.

The problem with this construction is that $\BB$ and $\ZZ$ are not
Banach spaces, and our theory needs the Banach space structure. We
will describe how to work around this in Section
\ref{subsection:multi}.
\end{example}

\subsection{The jumble norm}

The classical cut-norm, which plays a key role in the limit theory of
bounded graphons, is unfortunately not well suited for this general
setting that allows for unbounded functions. We introduce a variant
that serves this goal better.

\begin{definition}
Let the \emph{jumble-norm} on the function space $\LL$ be defined by
\[
\|u\|_\pcut:=\sup_{\substack{S,T\subseteq[0,1]}}
\frac{1}{\sqrt{\lambda(S)\lambda(T)}}
\Bigl|\int\limits_{S\times T}u(x,y) dx\, dy\Bigr|
\]
(To motivate the name, we note that Thomason \cite{Tho} uses this
normalization in the definition of ``jumble graphs'', a version of
quasirandom graphs.) 

Also, let
\[
\|W\|_\pcut := \sup_{\substack{f\in\BB\backslash\{0\}\\\|f\|_\BB=1}}\bigl\|\langle f,W\rangle\bigr\|_\pcut.
\]
We remark that this can also be written as
\[
\sup_{S,T\subseteq[0,1]}\frac{1}{\sqrt{\lambda(S)\lambda(T)}}
\Bigl\|\int\limits_{S\times T}W(x,y) dx\, dy\Bigr\|_\ZZ,
\]
where the integral is to be taken in the weak-* sense, i.e., $\int\limits_{S\times T}W$ is the unique element $z\in\ZZ$ such that for any $f\in\BB$ one has $\langle f,z\rangle=\int\limits_{S\times T}\langle f,W\rangle$
 (cf. \cite[Chap. XII., Prop. 3.3]{Musial}).
\end{definition}

We note that for every $u\in\LL$ (in fact, for every $u\in
L_2([0,1]^2$), the supremum in the definition above is finite.
Indeed,
\begin{align}\label{EQ:BOX2L2}
\frac{1}{\sqrt{\lambda(S)\lambda(T)}}
\Bigl|\int\limits_{S\times T}u(x,y) dx\, dy\Bigr|&=
\frac{1}{\sqrt{\lambda(S)\lambda(T)}}
\bigl|\langle \one_{S\times T}, u\rangle\big|\nonumber\\
&\le \frac{1}{\sqrt{\lambda(S)\lambda(T)}}
\|\one_{S\times T}\|_2 \|u\|_2 = \|u\|_2.
\end{align}
A fortiori, we obtain $\|W\|_\pcut\leq \|W\|_2$.

Let us also compare the jumble norm with the cut norm
\[
\|u\|_\square:=\sup_{\substack{S,T\subseteq[0,1]}}
\Bigl|\int\limits_{S\times T}u(x,y) dx\, dy\Bigr|.
\]
It is easy to see that
\[
\|u\|_\square\le \|u\|_\pcut \le \|u\|_\square^{1/2} \|u\|_\infty^{1/2},
\]
showing that for bounded kernels the two norms define the same
topology.

In the case of stepfunctions, the sets attaining the supremum in the
definition can be chosen in a special way.

\begin{lemma}\label{LEM:STEP-CUT}
Let $\Ps=\{S_1,\ldots,S_k\}$ be a measurable partition of $[0,1]$,
and $u:[0,1]^2\to\R$ a stepfunction with steps in $\Ps\times\Ps$.
Then there exist $T_1=\bigcup_{i=1}^\alpha S_{a_i}$ and
$T_2=\bigcup_{j=1}^\beta S_{b_j}$ such that
\[
\|u\|_\pcut =\frac{1}{\sqrt{\lambda(T_1)\lambda(T_2)}}
\Bigl|\int\limits_{T_1\times T_2}u\Bigr|.
\]
\end{lemma}

\begin{proof}
First note that since $u$ is a stepfunction, given any $Q_1,Q_2\subseteq [0,1]$, the value of
\begin{align*}
\frac{1}{\sqrt{\lambda(Q_1)\lambda(Q_2)}}
\Bigl|\int\limits_{Q_1\times Q_2}u\Bigr|&=\frac{\left|\sum_{j_1,j_2=1}^k \lambda(Q_1\cap S_{j_1})\lambda(Q_2\cap S_{j_2}) u|_{(Q_1\cap S_{j_1})\times (Q_2\cap S_{j_2})}\right|}
{\left(\sum_{j=1}^k\lambda(Q_1\cap S_j)\right)\left(\sum_{j=1}^k\lambda(Q_2\cap S_j)\right)}\\
&=:\mc{F}\left((\lambda(Q_i\cap S_j))_{1\leq j\leq k, i\in\{1,2\}}\right)
\end{align*}
depends only on the $2k$ variables
\[
(\lambda(Q_i\cap S_j))_{1\leq j\leq k, i\in\{1,2\}}\in\left(\prod_{j=1}^k[0,\lambda(S_j)]\right)\times\left(\prod_{j=1}^k[0,\lambda(S_j)]\right),
\]
and this dependence is continuous. By compactness there then exist some $\alpha_{i,j}\in [0,\lambda(S_j)]$ ($1\leq j\leq k, i\in\{1,2\}$) at which $\mc{F}$ attains its maximum. Since the Lebesgue measure on $[0,1]$ is non-atomic, there actually exist measurable subsets $R_{i,j}\subset S_j$ with $\lambda(R_{i,j})=\alpha_{i,j}$ for $1\leq j\leq k, i\in\{1,2\}$. Letting $R_i:=\bigcup_{j=1}^k R_{i,j}$ ($i\in\{1,2\}$), we obtain
\begin{equation}\label{eqn:opt}
\|u\|_\pcut =\frac{1}{\sqrt{\lambda(R_1)\lambda(R_2)}}
\Bigl|\int\limits_{R_1\times R_2}u\Bigr|.
\end{equation}
Now assume that for some $1\leq \ell\leq k$
\[
0<\lambda(R_{1,\ell})<\lambda(S_\ell).
\]
Let $R_1':=R_1\backslash S_\ell$ and $R_1'':=R_1\cup S_\ell$, and set
$a:=\int_{R_1'\times R_2}u$, $d:=\lambda(R_2)>0$, $c:=d\lambda(R_1')\geq0$
and $b:=\int_{\{x_0\}\times R_2}u(x_0,\cdot)$ for some $x_0\in
S_\ell$. Note that the value of $b$ is well-defined since $u$ is a
stepfunction. We may, without loss of
generality, assume that $a\geq0$. If we were to have $a>0, b<0$, then the integral of $u$ over $R_1'\times R_2$ and $R_{1,\ell}\times R_2$ would have different signs, so at least one of the choices $R_1'$ or $R_{1,\ell}$ instead of $R_1$ would lead to a strictly larger value than $\|u\|_\pcut$, which is a contradiction. Hence we may assume $b\geq 0$ as well.
Let now $S_\ell'$ be any subset
of $S_\ell$ of measure $\alpha$. With the notation
\[
N(\alpha):=\frac{1}{\sqrt{\lambda(R_1'\cup S_\ell')\lambda(R_2)}}
\Bigl|\int\limits_{(R_1'\cup S_\ell')\times R_2}u\Bigr|
=\frac{a+b\alpha}{\sqrt{c+d\alpha}}
\]
we have that
\[
N'(\alpha)=\frac{bd\alpha+(2bc-ad)}{2(c+d\alpha)^{3/2}},
\]
i.e., if $bd > 0$, the function $N$ is strictly monotone decreasing
until its minimum, then strictly monotone increasing, and hence if
$N(0)<N(\lambda (R_1\cap S_\ell))$, then also $N(\lambda (R_1\cap
S_\ell))< N(\lambda(S_\ell))$. If $bd=0$, then at least one of
$N(0)\geq N(\lambda (R_1\cap S_\ell))$ and $N(\lambda (R_1\cap
S_\ell))\leq N(\lambda(S_\ell))$ is true.

But note that
\[
N(0)=\frac{1}{\sqrt{\lambda(R_1')\lambda(R_2)}}
\Bigl|\int\limits_{R_1'\times R_2}u\Bigr|,
\]
\[
N(\lambda(R_1\cap S_\ell))=\frac{1}{\sqrt{\lambda(R_1)\lambda(R_2)}}
\Bigl|\int\limits_{R_1\times R_2}u\Bigr|,
\]
and
\[
N(\lambda(S_\ell))=\frac{1}{\sqrt{\lambda(R_1'')\lambda(R_2)}}
\Bigl|\int\limits_{R_1''\times R_2}u\Bigr|.
\]
This means that we may choose $R_1$ so that it contains either all of
$S_\ell$, or none of it, whilst still satisfying (\ref{eqn:opt}).
Iterating for all of the $S_i$ we obtain the desired $T_1$, and
repeating for $R_2$ yields us $T_2$.
\end{proof}

The next lemma shows how weighted integrals can be estimated in terms
of the jumble norm.

\begin{lemma}\label{lem:K}
Let $u\in L^2([0,1]^2)$ and $f,g\in L^3([0,1])$. Then
\begin{equation*}
\Bigl|\int\limits_{[0,1]^2} u(x,y)f(x)g(y) dx\, dy \Bigr|\le
8 \|u\|_\pcut \|f\|_3\|g\|_3.
\end{equation*}
\end{lemma}

\begin{proof}
First, we prove the stronger inequality
\begin{equation}\label{EQ:POS}
\Bigl|\int\limits_{[0,1]^2} u(x,y)f(x)g(y) dx\, dy \Bigr|\le
2 \|u\|_\pcut \|f\|_3\|g\|_3
\end{equation}
for the case when $f,g\ge 0$. Writing $f(x) = \int\limits_0^\infty
\one(t\le f(x))\,dt$ and $g(x) = \int\limits_0^\infty \one(s\le
g(x))\,ds$, we have
\begin{align*}
\int\limits_{[0,1]^2} &u(x,y)f(x)g(y)\,dx\,dy = \int\limits_0^\infty
\int\limits_0^\infty \int\limits_{[0,1]^2} u(x,y)\one(f(x)\ge t) \one(g(y)\ge
s)\,dx\,dy\,dt\,ds\\
&\le \int\limits_0^\infty \int\limits_0^\infty \|u\|_\pcut
\sqrt{\lambda\{x:f(x)\ge t\} \lambda\{y:g(y)\ge s\}} \,dt\,ds\\
&= \|u\|_\pcut \Bigl(\int\limits_0^\infty\sqrt{\lambda\{x:f(x)\ge t\}} \,dt\Bigr)
\Bigl(\int\limits_0^\infty\sqrt{\lambda\{y:g(y)\ge s\}} \,ds\Bigr).
\end{align*}
To estimate the integrals on the right, let $h$ denote the monotone
decreasing reordering of $f$. Then, using H\"older's Inequality,
\begin{align*}
\int\limits_0^\infty\sqrt{\lambda\{x:f(x)\ge t\}} \,dt
&=\int\limits_0^\infty \sqrt{h^{-1}(t)} \,dt = \int\limits_0^1 h(x^2)\,dx =
\int\limits_0^1 \frac1{2\sqrt{y}}h(y)\,dy\\
&\le \frac12 \Bigl\|\frac1{\sqrt{y}}\Bigr\|_{3/2}\|h\|_3
= 2^{1/3}\|f\|_3.
\end{align*}
Using a similar estimate for $g$, then repeating for the function $-u$ we obtain \eqref{EQ:POS}.

In the general case, let $f^+$ and $g^+$ denote the positive parts,
and $f^-$ and $g^-$ the negative parts of the functions $f$ and $g$,
respectively, so that $f=f^+-f^-$ and $g=g^+-g^-$. Then
\begin{align*}
\Bigl|\int\limits_{[0,1]^2}& u(x,y)f(x)g(y) dx\, dy \Bigr|\\
&\leq \left|\,\int\limits_{[0,1]^2} u(x,y)f^+(x)g^+(y) dx\,dy\right|
+\left|\,\int\limits_{[0,1]^2} u(x,y)f^+(x)g^-(y) dx\,dy\right|\\
&+\left|\,\int\limits_{[0,1]^2} u(x,y)f^-(x)g^+(y) dx\,dy\right|
+\left|\,\int\limits_{[0,1]^2} u(x,y)f^-(x)g^-(y) dx\,dy\right|
\end{align*}
Each term can be estimated by \eqref{EQ:POS}, and using the trivial
facts that $\|f^+\|_3, \|f^-\|_3\le \|f\|_3$ and $\|g^+\|_3,
\|g^-\|_3\le \|g\|_3$, the lemma follows.
\end{proof}

\begin{remark}\label{REM:K-NORM}
A more careful computation would improve the factor of $8$ to $4$.
One can strengthen the lemma in other ways as well. First, instead of
the $L^3$-norms on the right side, we could use the $L^p$-norm for
any $p>2$ (but not $p=2$). Second, for $f\in L^p[0,1]$ ($p>2$) we can
introduce the functional
\[
K(f):=\int\limits_0^\infty \sqrt{\lambda\{x: |f(x)|\ge t\} }\,dt.
\]
It is not hard to see that $K(.)$ is a norm on $L^p(0,1)$ for $p>2$ (it is in fact the Lorentz norm on the larger space $L^{2,1}$),
and it satisfies
\[
K(f) \le 2^{-1/p}\Bigl(\frac{p-1}{p-2}\Bigr)^{(p-1)/p} \|f\|_p.
\]
In terms of the norm $K$, the conclusion of Lemma \ref{lem:K} can be
strengthened (with the same proof):
\[
\Bigl|\int\limits_{[0,1]^2} u(x,y)f(x)g(y) dx\, dy \Bigr|\leq
4 \|u\|_\pcut K(f)K(g).
\]
\end{remark}

\subsection{Counting Lemma}

Our next goal is to prove appropriate generalizations of the Counting
Lemma from bounded kernels to unbounded ones (see Lemma 10.24 in
\cite{Hombook}). It is clear that
\[
|t(K_2,w)-t(K_2,w')| = \Bigl|\int_{[0,1]^2} (w-w')\Bigr| \le
\|w-w'\|_\pcut.
\]
The following lemma generalizes this inequality to densities of other
graphs.

\begin{lemma}\label{lem:DecDist}
Let $F$ be a simple graph with $l\ge 2$ edges, and let $w$ and $w'$
be $\LL$-decorations of $F$. For $a\in E(F)$, define
\[
M_a:=\max \bigl\{\|w_a\|_{3l-3},
\|w'_a\|_{3l-3} \bigr\}.
\]
Then we have
\begin{eqnarray*}
|t(F,w)-t(F,w')|\leq 8\sum_{a\in E(F)}
\|w_a-w'_a\|_\pcut \prod_{b\in E(F)\setminus\{a\}} M_b.
\end{eqnarray*}
\end{lemma}

If we define $\|w\|_0=1$ for every function $w$, then the lemma
remains valid for all graphs $F$ with at least one edge.

\begin{proof}
Let $V(F)=[k]$. It suffices to show that if $w'$ is obtained from $w$ by
changing the decoration of a single edge $a=uv$, then
\[
|t(F,w)-t(F,w')|\le 8\|w_a-w_a'\|_\pcut \,\prod_{b\in E(F)\setminus\{a\}} M_b.
\]
We may assume that $u=1$ and $v=2$. Then
\begin{equation}\label{eqn:telescopic}
t(F,w)-t(F,w')=\int\limits\limits_{[0,1]^k} \bigl(w_a(x_1,x_2)-
w'_a(x_1,x_2)\bigr)\prod_{ij\not=e} w_{ij}(x_i,x_j)\,dx.
\end{equation}
We can break the product into two parts:
\[
\prod{}'=\prod_{ij\in E(F), i=1,j>2},\qquad
\prod{}''=\prod_{ij\in E(F), i,j>1}.
\]
Let us fix $x_3,\dots,x_k$, and integrate just with respect to $x_1$
and $x_2$. Then $f=\prod{}'$ depends only on $x_1$ and $g=\prod{}''$
depends only on $x_2$, and clearly $f,g\in L^p(0,1)$ for every
$p\ge1$. So we get by Lemma \ref{lem:K} that
\begin{align*}
\Bigl|\int\limits_{[0,1]^2} &\bigl(w_a(x_1,x_2)- w'_a(x_1,x_2)\bigr)
f(x_1)g(x_2)\,dx_1\,dx_2\Bigr|\\ &\le 8 \|w_a-w_a'\|_\pcut \|f\|_3\|g\|_3.
\end{align*}
Integrating this with respect to $x_3,x_4,\dots$, we get by
H\"older's Inequality
\begin{align*}
|t(F,w)-t(F,w')|&\le 8\|w_a-w_a'\|_\pcut \int\limits_{[0,1]^{k-2}}
\left(\int\limits_0^1 \prod{}'
|w_{ij}(x_i,x_j)|^3 \,dx_1\right)^{1/3}\\ &\times \left(\int\limits_0^1 \prod{}''
|w_{ij}(x_i,x_j)|^3\,dx_2\right)^{1/3} dx_3\ldots dx_k\\
&\le 8\|w_a-w_a'\|_\pcut \left(\int\limits_{[0,1]^k}
\prod{}' |w_{ij}(x_i,x_j)|^3 dx_1\ldots dx_k\right)^{1/3}\\
&\times \left(\int\limits_{[0,1]^k}
\prod{}'' |w_{ij}(x_i,x_j)|^3 dx_1\ldots dx_k\right)^{1/3}\\
&\le 8\|w_a-w_a'\|_\pcut \prod_{ij\not=e} \|w_{ij}\|_{3l-3}\le
8\|w_a-w_a'\|_\pcut \prod_{b\not=a} M_b.
\end{align*}
This proves the Lemma.
\end{proof}

As a special case, we get the following version of the Counting
Lemma.

\begin{cor}[\bf Counting Lemma for Unbounded Kernels]\label{COR:COUNT-UB}
Let $u,w\in\LL$ and let $F$ be a simple graph with $l\ge 1$ edges.
Then
\[
|t(F,u)-t(F,w)| \le 8l\|u-w\|_\pcut
\max\bigl(\|u\|_{3l-3},\|w\|_{3l-3}\bigr)^{l-1}.
\]
\end{cor}

As a further corollary, we can estimate the difference between the
homomorphism densities of a decorated graph in two $\ZZ$-graphons
with the help of distances in jumble norm.

\begin{cor}\label{prop:DecDist}
Let $\Fb=(F,f)$ be a $\BB$-decorated graph with $k$ nodes and $l$
edges, and let $U,W\in\WW_\ZZ$. Then
\begin{eqnarray*}
|t(\Fb,U)-t(\Fb,W)|\leq 8l\,\|f\|_\infty^{l}\,\|U-W\|_\pcut\,
\max\bigl(\|U\|_{3l-3}, \|W\|_{3l-3}\bigr)^{l-1}.
\end{eqnarray*}
\end{cor}
\begin{proof}
Let $u_a:=\langle f_a,U\rangle$, $w_a:=\langle f_a,W\rangle$ and $M_a:=\max \bigl\{\|u_a\|_{3l-3},
\|w_a\|_{3l-3} \bigr\}$ for $a\in E(F)$. Then by Lemma \ref{lem:DecDist} we have that
\begin{align*}
|t(\Fb,U)-t(\Fb,W)|=&|t(F,(u_a)_{a\in E(F)})-t(F,(w_a)_{a\in E(F)})|\\
\leq& 8\sum_{a\in E(F)}
\|u_a-w_a\|_\pcut \prod_{b\in E(F)\setminus\{a\}} M_b\\
\leq& 8l\,\|f\|_\infty\|U-W\|_\pcut \max\bigl(\|f\|_\infty\|U\|_{3l-3}, \|f\|_\infty\|W\|_{3l-3}\bigr)^{l-1}\\
=&
 8l\,\|f\|_\infty^{l}\,\|U-W\|_\pcut\,
\max\bigl(\|U\|_{3l-3}, \|W\|_{3l-3}\bigr)^{l-1}
\end{align*}
\end{proof}

Also, we can directly extract the following from the proof of Lemma \ref{lem:DecDist}.
\begin{cor}\label{prop:DecDist2}
Let $\Fb=(F,f)$ be an $\Fs$-decorated graph, $s,s'\in\LL^\Fs$ and
\[M_\psi:=\max\left\{\tensor*[]{\left\|s_\psi\right\|}{_{3l-3}},
\tensor*[]{\left\|s'_\psi\right\|}{_{3l-3}}\right\}\]
for each $\psi\in\ms{F}$. Then
\begin{eqnarray*}
|t(\Fb,s)-t(\Fb,s')|\leq 8\sum_{ij\in E(F)}
\left( \tensor*[]{\left\|s_{f_{ij}}-s'_{f_{ij}}\right\|}{_{\pcut }}
\prod_{\alpha\beta\in E(F)\backslash \{ij\}} M_{f_{\alpha\beta}}\right)
\end{eqnarray*}
\end{cor}

We conclude this section with a simpler version of the counting
lemma, using the $L^l$ norm, rather than the jumble norm.

\begin{lemma}\label{LEM:COUNT-LP}
Let $\Fb=(F,f)$ be a $\BB$-decorated graph with $k$ nodes and $l$
edges, and let $U,W\in\WW_\ZZ$. Then
\[
|t(\Fb,U) - t(\Fb, W)| \le l\, \|f\|_\infty^l\, \|U-W\|_{l}
\max\bigl(\|U\|_l,\|W\|_l\bigr)^{l-1}.
\]
\end{lemma}
\begin{proof}
 The proof follows the ideas of Lemma \ref{lem:DecDist} and Corollary \ref{prop:DecDist}, but applies Hölder's Inequality directly after equation \ref{eqn:telescopic}.
\end{proof}

\subsection{The Weak Regularity Lemma}\label{sect:weak_reg}

We generalize the Regularity Lemma to our setting, based on the
notion of the jumble norm. The proof is a straightforward
generalization from the bounded case, but since we work with a
different norm here, we include it for completeness.

The {\it stepping operator} associated with a measurable partition
$\PP=\{S_1,\dots,S_k\}$ of $[0,1]$ assigns to every function
$u\in\LL$ the stepfunction $u_\PP$, where
\[
u_\PP(x,y) = \frac1{\lambda(S_i)\lambda(S_j)}
\int\limits_{S_i\times S_j} u(x,y)\,dx\,dy\quad\text{for }x\in S_i, y\in S_j,
\]
with the convention of assigning value $0$ whenever one of the sets has measure $0$.\\
Also, again using weak-* integrals, the stepping operator can be extended to $\ZZ$-graphons as
\[
W_\PP(x,y) := \frac1{\lambda(S_i)\lambda(S_j)}
\int\limits_{S_i\times S_j} W(x,y)\,dx\,dy\quad\text{for }x\in S_i, y\in S_j
\]

We note that by definition of the weak-* integral, the stepping operator commutes with applying a linear
functional: for every $W\in\WW_\ZZ$ and $f\in\BB$, we have
\begin{equation}\label{EQ:STEP-LIN}
\langle f, W\rangle_\PP =\langle f, W_\PP\rangle.
\end{equation}

Let us summarize some basic properties of this operator; analogues of
these for the cut-norm and for bounded functions were proved in
\cite{Hombook}.

\begin{prop}\label{prop:StepContr}
The stepping operators are contractive in the jumble norm.
\end{prop}

\begin{proof}
Let $u\in\LL$ and let $\PP$ be a measurable partition. Let $S$ and
$T$ be sets attaining the supremum in the definition of
$\|u_\PP\|_\pcut$, which by Lemma \ref{LEM:STEP-CUT} can be assumed
to be unions of partition classes of $\PP$. Then
\[
\|u_\PP\|_\pcut =
\frac{1}{\lambda(S)\lambda(T)}\Bigl|\int\limits_{S\times T}
u_\PP(x,y)\,dx\,dy\Bigr| =
\frac{1}{\lambda(S)\lambda(T)}\Bigl|\int\limits_{S\times T}
u(x,y)\,dx\,dy\Bigr| \le \|u\|_\pcut.
\]
\end{proof}

It is easy to check that the stepping operator is contractive with
respect to the $L^p$-norm of functions in $L^p([0,1]^2)$ for all
$p\ge 1$. In fact, it is contractive with respect to any
``reasonable'' norm (cf. Proposition 14.13 in \cite{Hombook}).

The stepfunction $u_\PP$ is the best approximation of $u$ in the
$L^2$-norm among all stepfunctions with the same steps. This is no
longer true if the $L^2$-norm is replaced by the jumble norm, but it
is true up to a factor of $2$, as shown by the following
straightforward generalization of an observation of Frieze and Kannan
\cite{FK}.

\begin{prop}\label{prop:PartRef}
Let $v\in\LL$ be a stepfunction and let $\PP$ be the partition of
$[0,1]$ that is finer than the partition into the steps of $v$. Then
for any $u\in\LL$ we have
\[
\bigl\|u-u_\PP\bigr\|_\pcut \leq
2\bigl\|u-v\bigr\|_\pcut.
\]
\end{prop}

\begin{proof}
We have $v=v_\PP$, and hence by Proposition \ref{prop:StepContr},
\begin{align*}
\|u-u_\PP\|_\pcut &\le \|u-v\|_\pcut + \|v-u_\PP\|_\pcut =
\|u-v\|_\pcut + \|v_\PP-u_\PP\|_\pcut\\
&= \|u-v\|_\pcut + \|(v-u)_\PP)\|_\pcut\le 2\|u-v\|_\pcut.
\end{align*}
\end{proof}

After this preparation, we are able to state the main lemma in this
section.

\begin{lemma}[\bf Weak Regularity Lemma for Unbounded Kernels]
\label{LEM:W-REG-GRAPHON-UB} For every symmetric function $w\in
L_2([0,1]^2)$ and every $k\ge 1$ there is a partition $\PP$ of
$[0,1]$ into $k$ measurable sets such that
\[
\|w-w_\PP \|_\pcut \le \frac4{\sqrt{\log k}}\|w\|_2.
\]
\end{lemma}

In the case of bounded kernels (or simple graphs), one may require
that the partition $\PP$ is an equipartition (allowing a little
larger error). This is, however, no longer true in the unbounded
setting: parts where the function is large must be partitioned into
more pieces.

\begin{proof}
Let
\[
j=\Bigl\lfloor \frac{{\log k}}2\Bigr\rfloor\quad\text{ and }\quad
c= \frac{2j+2}{{\log k}}>1.
\]
Let $S$ and $T$ be measurable subsets of $[0,1]$ such that
\[
\|w\|_\pcut\le \frac{c}{\sqrt{\lambda(S)\lambda(T)}}
\Bigl|\int\limits_{S\times T}w\Bigr|
= c\frac{\langle w, \one_{S\times T}\rangle}{\|\one_{S\times T}\|_2}.
\]
Let $a:=\langle w,\one_{S\times T}\rangle/\|\one_{S\times T}\|_2^2$.
Clearly $a \ge \|w\|_\pcut/(c\|\one_{S\times T}\|_2)$, and so
\begin{equation}\label{EQ:SZEM-ONE-STEP}
\|w-a\one_{S\times T}\|^2_2 = \|w\|_2^2
- a^2\|\one_{S\times T}\|_2^2 \le \|w\|_2^2 -\frac{1}{c^2}\|w\|_\pcut^2.
\end{equation}

Applying this observation repeatedly, we get pairs of sets $S_i,T_i$
and real numbers $a_i$ such that for all $r=1,2,\dots$, the
``remainder'' $w_r=w-\sum_{i=1}^r a_i\one_{S_i\times T_i}$ satisfies
\[
\|w_r\|_2^2\le \|w\|^2_2 - \sum_{i=0}^{r-1} \frac{1}{c^2}
\|w_i\|_\pcut^2.
\]
Applying this inequality with $r=j$, using \eqref{EQ:BOX2L2} and
$c>1$, it follows that
\[
\sum_{i=0}^j \|w_i\|_\pcut^2\le c^2\|w\|^2_2.
\]
Hence there is a $0\le m\le j$ with
\[
\|w_m\|_\pcut \le \frac{c}{\sqrt{j+1}} \|w\|_2 =  \frac{2\sqrt{j+1}}{\log k}\|w\|_2\leq \frac{2\sqrt{2j}}{\log k}\|w\|_2\leq \frac{2}{\sqrt{\log k}}\|w\|_2,
\]
which means that the stepfunction $v=\sum_{i=1}^m a_i\one_{S_i\times
T_i}$ satisfies the inequality
\[
\|w-v\|_\pcut \le \frac2{\sqrt{\log k}}\|w\|_2.
\]
The stepfunction $v$ has at most $2^j\le\sqrt{k}$ steps. It may not
be symmetric, but we can replace it by $(v(x,y)+v(y,x))/2$, which has
at most $4^j\le k$ steps. By Proposition \ref{prop:PartRef}, we can
replace $v$ by $w_\PP$ (where $\PP$ is the partition into the steps
of $v$) at the cost of doubling the error.
\end{proof}

\begin{lemma}[\bf Weak Regularity Lemma for several Unbounded Kernels]
\label{lem:WeakReg} For every $t\in\N^*$, $\e\in(0,\infty)$,
$u_1,u_2,\ldots,u_t\in\LL$ and measurable partition $\Ps_1$ of
$[0,1]$ there exists a positive integer $q=q(|\Ps_1|,t,\e)$ and a
measurable $q$-partition $\Ps_2$ of $[0,1]$ refining $\Ps_1$ such
that
\[
\left\|u_i-(u_i)_{\Ps_2}\right\|_\pcut \leq \e \|u_i\|_2
\]
for each $1\leq i\leq t$.
\end{lemma}

\begin{proof}
Using Lemma \ref{LEM:W-REG-GRAPHON-UB} we obtain measurable
partitions $\Rs_1,\ldots,\Rs_t$ of $[0,1]$ into
$2^{64\lceil\frac{1}{\e^2}\rceil}$ sets such that for each $1\leq i\leq
t$ we have
\[
\left\|u_i-(u_i)_{\Rs_i}\right\|_\pcut \leq
\frac{\e}{2} \|u_i\|_2.
\]
Let $\Ps_2$ be the common refinement of $\Ps_1$ and all of the
$\Rs_i$'s. Then
\[
|\Ps_2|=|\Ps_1|\cdot\prod_{i=1}^t
|\Rs_i|=|\Ps_1|\cdot
2^{4t\lceil\frac{1}{\e}\rceil}=:q(|\Ps_1|,t,\e),
\]
 and by Proposition
\ref{prop:PartRef}
\[
\left\|u_i-(u_i)_{\Ps_2}\right\|_\pcut \leq
2\left\|u_i-(u_i)_{\Rs_i}\right\|_\pcut\leq
\e\|u_i\|_2
\]
for each $1\leq i\leq t$.
\end{proof}

\ignore{\it Nincs-e vajon valamilyen altalanositas $\WW_\ZZ$-ben valo
kozelitesre?}

\section{Convergence of Banach space valued graphons}

\subsection{Dense sets}

In this preliminary section we prove two lemmas which will allow us
to use countable ``test sets'' for certain properties.
The first lemma lets us prove weak-* measurability using only a countable dense subset of $\BB$.

\begin{lemma}\label{le:weak*}
Let $\Fs\subset\BB$ be a countable dense subset, and let $W:[0,1]^2\to\ZZ$ be a function such that $\langle f,W\rangle$ is measurable for each $f\in\Fs$. Then $W$ is weak-* measurable.
\end{lemma}
\begin{proof}
Since $\Fs$ is countable and dense in $\BB$, for any $b\in\BB$ there exists a sequence $(f_n)\subset\Fs$ that converges to $b$ in norm. But then $\langle b,W\rangle$ is the pointwise limit of the measurable functions $\langle f_n,W\rangle$, and hence itself measurable.
\end{proof}

Next we show that when dealing with convergence of homomorphism
densities of decorated graphs, it is enough to consider decorations
from a generating subset of the original decoration space.

\begin{lemma}\label{PROP:GENERATE}
Let $W_1,W_2\ldots\in\WW_\ZZ$ be $\ZZ$-graphons and $\Fs\subseteq\BB$
be a generating subset. Assume that the sequence $(\|W_n\|_p)$ is
bounded by some $0\leq c_p$ for each $1\leq p<\infty$.
Then the following are equivalent:

\smallskip

{\rm(i)} For every $\BB$-decorated graph $\Fb$, the sequence
    $t(\Fb,W_n)$ is convergent;

\smallskip

{\rm(ii)} For every $\Fs$-decorated graph $\Fb$, the sequence
    $t(\Fb,W_n)$ is convergent.
\end{lemma}

\begin{proof}
Let us assume (ii) holds. Then, by multilinearity with respect to the
decoration, the convergence also holds for every
$\lin(\Fs)$-decorated graph. Let $\Fb=(F,f)$ be a $\BB$-decorated
graph with $l\ge 1$ edges, and let $f_1,\ldots,f_{m}$ be the elements
present in the decoration of $\Fb$ for an appropriate $1\leq m\leq
l$. For a given $k\in\N^*$ and for each $1\le r\leq m$, let $f_r^k\in
\lin(\Fs)$ be such that
\begin{equation}\label{eqn:close}
\left\|f_r-f_r^k\right\|\leq \frac{1}{k},
\end{equation}
and let $\Fb^k$ denote the decorated graph obtained from $\Fb$ by
replacing each $f_r$ in its decoration by $f_r^k$.

Now fix $n\in\N^*$. Let $s\in\LL^\BB$ be defined by $s_f:=\langle f,
W_n\rangle$, and let $s'\in\LL^\BB$ be such that $s'_{f_r}=\langle
f_r^k, W_n\rangle$ for each  $1\leq r\leq m$, and $s'_f=s_f$
otherwise. Let
\[
d_{f_{\alpha\beta}}=\max\left\{\left\|s{_{f_{\alpha\beta}}}\right\|_{3l-3},
\left\|s{^\prime_{f_{\alpha\beta}}}\right\|_{3l-3}\right\},
\]
Then we have by definition $t(\Fb,W_n)=t(\Fb,s)$ and
$t(\Fb^k,W_n)=t(\Fb,s')$, and applying Corollary \ref{prop:DecDist2}
we obtain
\begin{align*}
\Bigl| t(\Fb,W_n)-t(\Fb^k,W_n)\Bigr|&=|t(\Fb,s)-t(\Fb,s')|\\
&\le 8\sum_{ij\in E(F)}\left( \tensor*[]{\left\|\tensor*[]s{_{f_{ij}}}-
\tensor*[]s{^\prime_{f_{ij}}}\right\|}{_{\pcut }}
\prod_{\alpha\beta\in E(F)\backslash ij} \tensor*[]d{_{f_{\alpha\beta}}}\right)\\
&\le 8 \sum_{ij\in E(F)}\left( \tensor*[]{\left\|\tensor*[]s{_{f_{ij}}}-
s'_{f_{ij}}\right\|}{_2}\prod_{\alpha\beta\in E(F)\backslash ij}
d_{f_{\alpha\beta}}\right)\\
&=8 \sum_{ij\in E(F)}\left( \left\|\langle f_{ij},W_{n}\rangle -
\langle f^k_{ij},W_{n}\rangle\right\|_2\prod_{\alpha\beta\in E(F)\backslash ij}
d_{f_{\alpha\beta}}\right)\\
&\le 8 \sum_{ij\in E(F)}\left( \left\|\frac{W_n}{k}\right\|_2
\prod_{\alpha\beta\in E(F)\backslash ij}
d_{f_{\alpha\beta}}\right)\\
&\le 8 \sum_{ij\in E(F)}\left(\frac{\tensor*[]c{_2}}{k}\prod_{\alpha\beta\in E(F)\backslash ij}
\left(\left\|\langle f_{\alpha\beta},W_{n}\rangle\right\|_{3l-3}+
\left\|\frac{W_n}{k}\right\|_{3l-3}\right)\right)\\
&\le8 \sum_{ij\in E(F)}\left(\frac{\tensor*[]c{_2}}{k}
\prod_{\alpha\beta\in E(F)\backslash ij}
\left(\left\|f_{\alpha\beta}\right\|+\frac{1}{k}\right)\tensor*[]c{_{3l-3}}\right)
\end{align*}
using inequality (\ref{eqn:close}) and pointwise estimates for the
weak-* evaluations of $W_n$. The last line is independent of $n$ and
converges to 0 as $k$ goes to infinity. Since the sequences
$t(\Fb^k,W_n)$ are convergent for each $k$, the convergence thus
follows for the sequence $t(\Fb,W_n)$.

The other implication obviously holds.
\end{proof}

\subsection{Moment sequences}

Let $\Fs$ be a countable generating subset of $\BB$.

\begin{definition}
The $\Fs$-\emph{moment sequence} of an element $z\in\ZZ$ is the
family $(\langle f,z \rangle:~f\in\Fs)$ of real numbers. The
$\Fs$-\emph{moment function sequence} of a function
$U:~[0,1]^2\to\ZZ$ is the family $(\langle f,U\rangle:~f\in\Fs)$ of
functions.
\end{definition}

\begin{prop}\label{PROP:MOM-MEAS}
A family of measurable symmetric functions $u_f:~[0,1]^2\to\R$
$(f\in\Fs)$ is the $\Fs$-moment sequence of a symmetric weak-* measurable function $W:~[0,1]^2\to\ZZ$ if and only if
$(u_f(x,y):f\in\Fs)$ is the $\Fs$-moment sequence of some element of
$\ZZ$ for every $(x,y)\in[0,1]^2$.
\end{prop}

\begin{proof}
The ``only if'' part is trivial. To prove the ``if'' part, notice
that by assumption we may define $W$ pointwise in a unique way.
Indeed, by the definition of an $\Fs$-moment sequence, for every
$(x,y)\in[0,1]^2$ there is a $z\in\ZZ$ such that $u_f(x,y)=\langle
f,z\rangle$ for every $f\in\Fs$. Thus element $z$ is uniquely
determined; indeed, if $z_1$ and $z_2$ are two such elements, then by
linearity $\langle g,z_1\rangle=\langle g,z_2\rangle$ for every $g\in
\lin(\Fs)$, and then by the density of $\lin(\Fs)$ in $\BB$ we have
$\langle b,z_1\rangle =\langle b,z_2\rangle$ for every $b\in\BB$,
implying that $z_1=z_2$. Thus we can define $W(x,y)=z$. Uniqueness of
$z$ also implies that $W$ is symmetric.

We have to show that this function $W$ is weak-* measurable as a
function $[0,1]^2\to \ZZ$. Let $\Gs:=\lin_{\Q }\Fs$, then $\Gs$ is a
dense countable set in $\BB$. By linearity, for every $g\in\Gs$ the function $\langle g,W\rangle$ is measurable. By Lemma \ref{le:weak*}, the function $W$ is then weak-* measurable.
\end{proof}

Our next lemma says that, under appropriate conditions, the limit of
a sequence of $\Fs$-moment function sequences of $\ZZ$-graphons is the
$\Fs$-moment function sequence of a $\ZZ$-graphon.

\begin{lemma}\label{LEM:LIM-MOMENT}
Let $W_n\in\WW_{\ZZ}$ ($n\in\N^*$) and suppose that there are
constants $c_p>0$ such that $\|W_n\|_p\le c_p$ for every $1\le
p<\infty$. Let $u_{n,f}=\langle f,W_n\rangle$ for $n\in\N^*$ and
$f\in\Fs$, and suppose that for every $f\in\Fs$, $u_{n,f}(x,y)\to
u_f(x,y)$ for almost all $(x,y)\in[0,1]^2$ for some function $u_f$. Then there exists a
$\ZZ$-graphon $W$ such that $\|W\|_p\le c_p$ for every $1\le
p<\infty$, and $u_f=\langle f,W\rangle$ for all $f\in\Fs$.
\end{lemma}

\begin{proof}
First we note that the condition that $\|W_n\|_p$ is bounded implies
that for almost all points $(x,y)\in[0,1]^2$, the sequence
$\|W_n(x,y)\|_\ZZ$ does not tend to infinity. We also know that for
almost all $(x,y)\in[0,1]^2$, $u_{n,f}(x,y)\to u_f(x,y)$ for all
$f\in\Fs$ (in particular $u_f$ is measurable). Let us call a point $(x,y)$ satisfying these conditions
{\it ordinary}.

Fix an ordinary point $(x,y)$. We can find a subsequence $\N_1$ of
the indices $n$ for which $\|W_n(x,y)\|_\ZZ$ remains bounded, and we
can select a further subsequence $\N_2$ for which the sequence
$(W_n(x,y):~n\in\N_2)$ is weak-* convergent. Let its limit be
$z\in\ZZ$. By the definition of weak-* convergence,
this limit satisfies
\begin{equation}\label{EQ:WEAKLIM}
\langle f,z\rangle=\lim_{n\in\N_2} \langle f,W_n(x,y)\rangle
=\lim_{n\in\N_2} u_{n,f}(x,y)
= u_f(x,y)
\end{equation}
for every $f\in\Fs$, and also
\begin{align}\label{eqn:NormBound}
\|z\|_\ZZ &= \sup_{b\in\BB\atop \|b\|_\BB=1}\langle b,z\rangle
= \sup_{b\in\BB\atop \|b\|_\BB=1}\lim_{n\in\N_2}\langle b,W_n(x,y)\rangle\nonumber\\
&\le\sup_{b\in\BB\atop \|b\|_\BB=1}\liminf_{n\in\N_2}\|b\|_\BB\|W_n(x,y)\|_\ZZ
= \liminf_{n\in\N_2} \|W_n(x,y)\|_\ZZ.
\end{align}

Just as in the previous proof, the conditions $\langle
f,z\rangle=u_f(x,y)$ determine the element $z\in\ZZ$ uniquely. It
follows that inequality \eqref{eqn:NormBound} holds for any weak-*
convergent norm-bounded subsequence $\N_2$. This implies that
\begin{equation}\label{eqn:NormBound2}
\|z\|_\ZZ \leq \liminf_{n\in\N} \|W_n(x,y)\|_\ZZ.
\end{equation}
Now Proposition \ref{PROP:MOM-MEAS} applies and yields a symmetric
function $W:~[0,1]^2\to\ZZ$ that is weak-* measurable, and
$u_f=\langle f,W\rangle$ for every $f\in\Fs$. We want to show that
$W$ is a $\ZZ$-graphon.

By the remark about uniqueness above, $W(x,y)$ must be equal to the
weak limit $z$ constructed above for every ordinary $(x,y)$, and setting it to $0$ at non-ordinary points we
hence have
\begin{equation}\label{eqn:NormBound3}
\|W(x,y)\|_\ZZ \leq \liminf_{n\in\N^*} \|W_n(x,y)\|_\ZZ.
\end{equation}
This can be written as
\[
0\leq (\|W(x,y)\|_\ZZ)^p\leq \liminf_{n\in\N^*}
(\|W_n(x,y)\|_\ZZ)^p.
\]

An application of Fatou's lemma then yields
that $(\|W\|_\ZZ)^p\in L^1([0,1]^2)$ for every $1\leq p<\infty$, and
its norm is bounded by $c_p^{p}$. This means that
\[
\|W\|_p = \Bigl(\bigl\|(\|W\|_\ZZ)^p\bigr\|_1\Bigr)^{1/p} \le c_p,
\]
showing that $W$ is a $\ZZ$-graphon with the required norm bounds.
\end{proof}

\subsection{Measure preserving transformations}

As a last step before turning our attention to proving the existence of a limit graphon for convergent sequences, it will be useful to touch upon the subject of uniqueness of the graphon representation. In other words, to what extent can we expect a graphon to be uniquely determined by the homomorphism densities? In the bounded, real valued case, this question can be fully answered, giving a characterization of all ``equivalent'' graphons, and this strongly ties into the fact that a number of convergence notions (density convergence, cut distance convergence, sampling convergence) turn out to be equivalent. In this more general setting, however, only partial results are known, and we refer to \cite{DKK} for details. In this section we shall concentrate on a sufficient condition under which two graphons exhibit the same densities, as this will play a pivotal role in the proof of our main theorem.

For a finite $\ZZ$-decorated graph $\Gb$ on the vertex set $[n]$, there is a natural way of identifying it with a $\ZZ$ graphon $W_{\Gb}$ that has the exact same densities. Namely, let
\[
W_{\Gb}(x,y):=g_{\lceil nx\rceil\lceil ny\rceil}.
\]
Then $W_{\Gb}$ is a stepfunction that essentially corresponds to the adjacency matrix of $\Gb$, and an easy computation shows that indeed $t(\Fb,\Gb)=t(\Fb,W_{\Gb})$ for any $\BB$-decorated graph $\Fb$. Now, note that for a finite graph $\Gb$, the value of  $t(\Fb,\Gb)$ does not change if we replace it by an isomorphic graph $\Gb'$ (i.e., we relabel its vertices). However, $W_{\Gb}\neq W_{\Gb}$, so this immediately shows that two different graphons may exhibit the same densities.\\
The relabeling of the vertices of $\Gb$ essentially amounts to a permutation of the $1/n$-length intervals corresponding to the steps of $W_{\Gb}$. Now note that for graphons, the homomorphism densities are defined via integrals, and in general, for any measure preserving map $\varphi:[0,1]\to[0,1]$ and any integrable function $H:[0,1]^k\to\mb{R}$, we have that
\[
\int_{[0,1]^k} H=\int_{[0,1]^k} H\circ \varphi,
\]
where $H\circ \varphi(x_1,x_2,\ldots,x_k):=H\circ \varphi^{\otimes k}(x_1,x_2,\ldots,x_k)=H(\varphi(x_1),\varphi(x_2),\ldots,\varphi(x_k)$. In particular, for any $\ZZ$-graphon $W$ and $\BB$-decorated graph $\Fb$ we have $t(\Fb,W)=t(\Fb,W\circ\varphi)$. Also, less trivially, we have the following invariance of the jumble norm as well.
\begin{lemma}\label{le:jumbleinvariance}
For any $u\in\mc{L}$ we have $\|u\|_\pcut=\|u\circ\varphi\|_{\pcut}$.
\end{lemma}
\begin{proof}
The proof works essentially as the proof for the cut norm in \cite[Lemma 5.5]{Janson}. The result is straightforward for any stepfunction due to Lemma \ref{LEM:STEP-CUT} and the fact that any measure preserving map is actually a measure preserving bijection between the finite atomic $\sigma$-algebras generated by the steps of non-zero measure of $u$ and $u\circ\varphi$, respectively.\\
For a general $u\in\mc{L}$, consider an arbitrary $\varepsilon>0$ and a stepfunction $v\in \mc{L}$ such that $\|u-v\|_2<\varepsilon$. Then we have the following:
\begin{align*}
&\left|\|v\|_\pcut-\|v\circ\varphi\|_{\pcut}\right|=0;\\
&\|u-v\|_\pcut\leq\|u-v\|_2<\varepsilon;\\
&\|u\circ\varphi-v\circ\varphi\|_\pcut\leq\|u\circ\varphi-v\circ\varphi\|_2=\|(u-v)\circ\varphi\|_2=\|u-v\|_2<\varepsilon.
\end{align*}
A few applications of the triangle inequality then imply $\left|\|u\|_\pcut-\|u\circ\varphi\|_{\pcut}\right|<2\varepsilon$.
\end{proof}

\subsection{Limit graphon}

The goal of this section is to prove our main result on the existence
of a limit graphon in the Banach space valued case. The proof follows the overall structure of the proof of the corresponding results in \cite{Hombook, LSz1}.

\begin{thm}\label{thm:Main2}
Let $\Fs$ be a countable generating subset of $\BB$ and
$W_1,W_2,\ldots$ be $\ZZ$-graphons satisfying the following
conditions:

\smallskip
{\rm(i)} for every $1\leq p<\infty$ there is a $c_p>0$ such that
    $\|W_n\|_p\le c_p$ for all $n$;

\smallskip
{\rm(ii)} the sequence $(t(F,W_n):~n\in\N^*)$ is convergent for
    every $\Fs$-decorated graph $F$.

\smallskip
\noindent Then there exists a $\ZZ$-graphon $W$ such that
$t(F,W_n)\to t(F,W)$ for every $\BB$-decorated graph $F$.
\end{thm}

\begin{proof}
Let $\Fs=\{f_1,f_2,\dots\}$ and $w_{n,m}=\langle f_m,W_n\rangle$. For
each $k\ge 1$, define $h(k)$ recursively by $h(1)=1$ and $h(k)=
q\bigl(h(k-1),k,1/(c_2k)\bigr)$, where $q$ is the function in Lemma
\ref{lem:WeakReg}. For every $k,n\ge1$, we construct a partition
$\PP_{n,k}$ of $[0,1]$ so that $|\PP_{n,k}|=h(k)$, $\PP_{n,k+1}$
refines $\PP_{n,k}$, and the stepfunctions
$u_{n,k,m}=(w_{n,m})_{\PP_{n,k}}$ satisfy
\begin{equation}\label{EQ:WU-REGLEMMA}
\bigl\|w_{n,m}-u_{n,k,m}\bigr\|_\pcut \le \frac1k \|f_m\|_{\BB}
\quad(m=1,\dots, k).
\end{equation}
Let $\PP_{n,k}=\{S_{n,k}^1,\dots,S_{n,k}^{h(k)}\}$. Let
$U_{n,k}=(W_n)_{\PP_{n,k}}$; then \eqref{EQ:STEP-LIN} implies that
$u_{n,k,m}=\langle f_m,U_{n,k}\rangle$, and by the contractivity of
the stepping operator, $\|U_{n,k}\|_p\le c_p$ for all $1\le
p<\infty$. Note that $u_{n,k,m}$ is defined for all $n,k$ and $m$,
but \eqref{EQ:WU-REGLEMMA} only holds for $m\le k$.

By selecting a subsequence $\N_1$ of the indices $n$, we may assume
that for every fixed $k$ and $1\le j\le h(k)$, the sequence
$\lambda(S_{n,k}^j)$ converges to a value $s_{k,j}$, and in fact it
converges fast enough so that $\sum_{n\in \N_1}
|\lambda(S_{n,k}^j)-s_{k,j}|<\infty$.\\
At this point, the issue is that even though the measures of the sets of a given index pair $(k,j)$ converge as $n$ tends to infinity, the sets themselves, pertaining to regularity partitions of different functions $W_n$, have no reason to exhibit any convergent behaviour. This is where the observations and Lemma \ref{le:jumbleinvariance} in the previous section become useful: since the homomorphism densities are not influenced by transforming the graphons via applying measure preserving transformations, nor are the various $L^p$ (condition (i)) or jumble norm (equation \ref{EQ:WU-REGLEMMA}) bounds, we could try to find an appropriate measure preserving transformation $\varphi_n:[0,1]\to[0,1]$ for each $n$ that would ``unscramble'' the sets $S_{n,k}^j$ and facilitate a natural limit transition for the partitions.\\
More precisely, there exist measure preserving transformations $\varphi_n:[0,1]\to[0,1]$ ($n\in\mb{N}^*$)
such that for every $k\in\mb{N}^*$ there is a measurable
partition $\PP_k=\{S_k^1,\dots,S_k^{h(k)}\}$ into intervals for which $\sum_{n\in\N_1} \lambda(S_k^j\triangle \varphi_n(S_{n,k}^j))<\infty$ for every $j$, and also
$\PP_{k+1}$ refines $\PP_k$.\\
Indeed, let $g_{n,k}:[0,1]\to\mb{N}^*$ be the measurable function defined via $g_{n,k}(x):=j$ whenever $x\in S^j_{n,k}$. Then the sequence $(g_{n,k}(x))_{k\in\mb{N}^*}$ corresponds to the sequence of indices of the partition classes $x$ belongs to. Now let
\[
g_{n}(x):=\sum_{k=1}^\infty \frac{g_{n,k}(x)}{\prod_{m=1}^k (h(m)+1)}.
\]
Note that $g_n:[0,1]\to[0,1]$ is a measurable function, and so there exists a measure preserving transformation $\varphi_n$  that yields a monotone increasing rearrangement of $g_n$ (e.g., $\varphi_n(x):=\lambda([g_n<g_n(x)])+\lambda([g_n=g_n(x)]\cap[0,x])$). As the denominators of the sum were chosen to grow fast enough, this will correspond to the ``lexicographic ordering'' of the points of $[0,1]$ with regards to the index sequence $(g_{n,k}(\cdot))$. Thus the partitions $\varphi_n\circ\PP_{n,k}$ are interval partitions, and the convergence of the sequences $\left(\lambda(S_{n,k}^j)\right)_{n\in\mb{N}_1}$ finishes the argument.

Since replacing the graphons $W_n$ by $W_n\circ\varphi_n$ in the original statement would lead to an equivalent claim, we may without loss of any generality assume that each $\varphi_n$ is the identity map, and avoid an extra layer of cumbersome notation.

Our next goal is to select a subsequence of the indices $n$ so that
the stepfunctions $u_{n,k,m}$ converge to some stepfunction with
$\PP_k$-steps almost everywhere. This is not obvious, as the sequence
$(u_{n,k,m})_{n\in\N_1}$ may not be uniformly bounded. To prove it, we
have to treat ``small'' partition classes separately.

Let $J_k\subseteq [h(k)]$ be the set of indices $j$ such that
$\lambda(S_k^j)>0$. Let $a_{n,k,m,i,j}$ be the value of $u_{n,k,m}$
on $S_k^i\times S_k^j$ (where $i,j\in [h(k)]$). Then we have that
\[
|a_{n,k,m,i,j}|\,\lambda(S_{n,k}^i)\,\lambda(S_{n,k}^j) \le
\|u_{n,k,m}\|_1 = \|\langle f_m,U_{n,k}\rangle\|_1\le
\|f_m\|_\BB \,\|U_{n,k}\|_1 \le c_1\|f_m\|_\BB,
\]
and since $\lambda(S_{n,k}^i)$ and $\lambda(S_{n,k}^j)$ both converge to a non-zero value, the sequence $(a_{n,k,m,i,j})_{n\in\mb{N}_1}$ remains
bounded. Hence using the usual diagonal method, we
can select subsequence $\N_2\subseteq\N_1$ so that $(a_{n,k,m,i,j})_{n\in\mb{N}_2}$
tends to some value $a_{k,m,i,j}$ for all $m$ and all $i,j\in J_k$.
We define $a_{k,m,i,j}=0$ if $i\notin J_k$ or $j\notin J_k$. (Note
that $\lambda(S_i\times S_j)=0$ in this case.) These values define a
stepfunction $u_{k,m}:~[0,1]^2\to\R$ with $\PP_k$-steps.

It is clear that $u_{n,k,m}(x,y)\to u_{k,m}(x,y)$ for $(x,y)\in
S_k^i\times S_k^j$ for $i,j\in J_k$, provided $x$ is contained in a
finite number of sets $S_k^i\triangle S_{n,k}^i$, and $y$ is
contained in a finite number of sets $S_k^j\triangle S_{n,k}^j$.
Since $\sum_n\lambda(S_k^j\triangle S_{n,k}^j)<\infty$, the
Borel-Cantelli Lemma implies that $u_{n,k,m}\to u_{k,m}$ almost
everywhere $(n\in\N_2,\ n\to \infty)$. Lemma \ref{LEM:LIM-MOMENT}
implies that for every $k$, the sequence $(u_{k,m})_{m\in\N^*}$ is the
$\Fs$-moment function sequence of some $\ZZ$-graphon $U_k$, and so
$u_{k,m}=\langle f_m,U_k\rangle$ and $\|U_k\|_p\le c_p$ for all $1\le
p<\infty$.\\
Note that we have a finite number of steps and $\lim_{n\in\N_2}\sum_{j=1}^{h(k)}\lambda(S_k^j\triangle S_{n,k}^j)=0$.
Also, whenever $i\not\in J_k$ or $j\not\in J_k$, we have that for each $1\leq p<\infty$,
\[
\left\|u_{n,k,m}|_{S_{n,k}^i\times S_{n,k}^j}\right\|_p\to 0.
\]
Indeed, if this were not the case for some $p$, then since the measures of the sets we restrict to tend to zero, the functions $u_{n,k,m}$ would not be uniformly bounded in $L^{q}$ for any $q>p$.
Thus, we have that for every $1\leq p<\infty$, $u_{n,k,m}\to u_{k,m}$ in $L^p$ ($n\in\N_2$, $n\to\infty$).

Now $(u_{n,r,m})_{\PP_{n,k}}=u_{n,k,m}$ for $r\ge k\ge m$, and we have a finite ($|J(k)|$) number of partition class sequences $(S^j_{n,k})_{n\in\mb{N}_2}$, each of which satisfy convergence in their measure to that of $S^j_k$, and a finite ($|J(k)|^2$) number of value sequences $(a_{n,k,m,i,j})_{n\in\mb{N}_2}$ each converging to the corresponding $a_{k,m,i,j}$. For the rest of the partition classes in $\mc{P}_k$, all measures and all assigned function values $a_{k,m,i,j}$ are zero, hence  it follows
that also $(u_{r,m})_{\PP_k}=u_{k,m}$. So we can apply the Martingale
Convergence Theorem to the sequence $(u_{k,m})_{k\in\N^*}$, and get that
$u_{k,m}$ tends to some symmetric function $u_m:~[0,1]^2\to\R$ both
almost everywhere and also in $L_1$ as $k\to\infty$.

Lemma \ref{LEM:LIM-MOMENT} implies that the sequence $(u_m)_{m\in\N^*}$
is the $\Fs$-moment function sequence of a $\ZZ$-graphon $U$ such
that $\|U\|_p\le c_p$ for all $1\le p<\infty$.\\
But since the sequence $(u_{k,m})_{k\in\N^*}$ is bounded in $L^p$ for every $1<p<\infty$ and $u_{k,m}\to u_m$ in $L^1$, it follows that the convergence also holds in $L^p$ for all $1\leq p<\infty$ (as above, lack of convergence for a given $p$ would imply no uniform boundedness in $L^q$ for $q>p$).

Next we show that for the subsequence of indices $n$ we selected, we
have $t(F,\ph,W_n)\to t(F,\ph,U)$ for every $\Fs$-decorated simple
graph $(F,\ph)$. To this end, let $\eps>0$ be given, and write
\begin{align*}
|t(F,\ph,W_n)-t(F,\ph,U)| \le &|t(F,\ph,W_n)-t(F,\ph,U_{n,k})|
+ |t(F,\ph,U_{n,k})-t(F,\ph,U_k)|\\ &+ |t(F,\ph,U_k)-t(F,\ph,U)|.
\end{align*}
We have a $q\ge 1$ such that every decoration in $(F,\ph)$ occurs
among $\{f_1,\dots,f_q\}$. By the definition of $\PP_{n,k}$, we have
\[
\bigl\|w_{n,m}-u_{n,k,m}\bigr\|_\pcut \le \frac1k \|f_m\|_{\BB}
\]
for every $m\le q$ and $k\ge q$. This means that
\[
\|\langle f_m,W_n\rangle-\langle f_m,U_{n,k}\rangle\|_\pcut
\le \frac1k \|f_m\|_{\BB},
\]
and by the Counting Lemma \ref{lem:DecDist}, we get that
$|t(F,\ph,W_n)-t(F,\ph,U_{n,k})|<\eps/3$ if $k$ is large enough
(independently of $n$).

Since $u_{k,m}\to u_m$ in $L^p$ for every $1\le p<\infty$, Lemma
\ref{LEM:COUNT-LP} implies that $t(F,\ph,U_k)\to t(F,\ph,U)$, and
hence $|t(F,\ph,U_k)-t(F,\ph,U)|<\eps/3$ if $k$ is large enough. Let
us fix $k$ so that both of these inequalities hold. Then, again by
Lemma \ref{LEM:COUNT-LP}, we get that $t(F,\ph,U_{n,k})\to
t(F,\ph,U_k)$, and hence $t|(F,\ph,U_{n,k})-t(F,\ph,U_k)|<\eps/3$ if
$n$ is large enough.

This proves that $t(F,\ph,W_n)\to t(F,\ph,U)$ for $n\in\N_2$. Since
the sequence $(t(F,\ph,W_n))$ is convergent, we have $t(F,\ph,W_n)\to
t(F,\ph,W)$ for $n\in\N^*$.

Finally, we show that the relation $t(F,\ph,W_n)\to t(F,\ph,W)$
$(n\in\N^*)$ holds not only for $\Fs$-decorations $\phi$ but also for
$\BB$-decorations $\ph$. First, it holds for $\lin(\Fs)$-decorations,
since $t(F,\ph,W)$ is multilinear in $\ph$ (where $\ph$ is considered
as an element in $\BB^{E(F)}$). Second, it holds for
$\BB$-decorations, since these can be approximated by
$\lin(\Fs)$-decorations, and $t(F,\ph,W)$ is continuous in $\ph$.
\end{proof}

\subsection{$W$-random decorated graphs}\label{sec:wrandom1}

Let $W$ be a $\ZZ$-graphon. For every $n\ge 1$, we can generate a
{\it $W$-random $\ZZ$-decorated graph} $\Gbb(n,W)$ on node set $[n]$
as follows: we select $n$ independent uniformly distributed points
$x_1,\dots,x_n$ from $[0,1]$, and label every pair $ij$ $(1\le i<j\le
n)$ by $W(x_i,x_j)$, and label every pair $ii$ $(1\le i\le
n)$ by $0\in\ZZ$.

\begin{thm}\label{THM:WRAND-CONV}
For every $\ZZ$-graphon and every $\BB$-decorated graph $\Fb$,
\[
t(\Fb,\Gbb(n,W))\to t(\Fb,W)\qquad(n\to\infty)
\]
with probability $1$.
\end{thm}

As an immediate corollary, we get:

\begin{cor}\label{THM:ZZ-RAND}
Every $\ZZ$-graphon is the limit of a convergent sequence of
$\ZZ$-decorated graphs.
\end{cor}

\begin{proof}
The proof is similar to the proof for the unweighted case in
\cite{LSz1} and in \cite{Hombook}, Section 11.2.1, but we have to be
somewhat more careful because of the unbounded values that occur.

Let us fix a countable generating set $\Fs\in\BB$, and an
$\Fs$-decorated graph $\Fb=(F,f)$ on the node set $[k]$. It is easy to
compute the expected injective subgraph densities in $\Gbb(n,W)$,
where $n\ge k$. We can generate a random map $[k]\to[0,1]$ by
composing a random map $x:~[n]\to[0,1]$ with a random injective map
$\ph:~[k]\to[n]$. This gives us that
\begin{align*}
t(\Fb,W) = \E_x\E_\ph \prod_{ij\in E(F)}
\langle f(ij), W(x_{\ph(i)},x_{\ph(j)}) \rangle
= \E_x \bigl(t_\inj(\Fb,\Gbb(n,W))\bigr).
\end{align*}
Lemma
\ref{LEM:T-INJ} implies that
\[
\E_x t(\Fb,\Gbb(n,W)) ) \to t(\Fb,W) \quad(n\to\infty).
\]

This shows that $W$-random graphs yield a sequence with the desired
subgraph densities {\it in expectation}. To show that these values
are concentrated, we need to estimate the fourth moments
(unfortunately, estimating the variance would not be quite enough). Sharper bounds can be obtained based on the theory of U-statistics, but the simpler argument below will be sufficient for us.
For $\ph:~V(F)\to[n]$, let $X_\ph=t_{\ph}(\Fb,\Gbb(n,W))-t(\Fb,W)$.
Then $\E_x(X_\ph)=0$ if $\ph$ is injective. Furthermore,
\[
(t(\Fb,\Gbb(n,W))-t(\Fb,W))^4 =
\frac1{n^{4k}} \sum_{\ph_1,\dots,\ph_4} X_{\ph_1}\dots X_{\ph_4},
\]
where $\ph_1,\dots,\ph_4$ range independently over all maps
$V(F)\to[n]$.

Let us take expectation in $x$. If the range of any of the $\ph_i$ is
disjoint from the others, then the expectation of $X_{\ph_i}$ can be
taken separately, and if, in addition, $\ph_i$ is injective, then
this expectation is $0$. So only those terms remain in which for
every $i$, the range of $\ph_i$ intersects the range of at least one
other $\ph_j$, or $\ph_i$ is not injective. This implies that the
range of $\ph_1\cup\dots\cup\ph_4$ has at most $4k-2$ elements, and
so the number of such terms $O(n^{4k-2})$. The expectation of such a
term is bounded by 
$2\Pi_{f}\cdot\|W\|_{4|E(F)|}^{4|E(F)|}$
, and hence
\[
\E_x\bigl((t(\Fb,\Gbb(n,W))-t(\Fb,W))^4\bigr)=O\left(\frac1{n^2}\right).
\]
This implies that for every fixed $\eps>0$,
\begin{align*}
\Pr\bigl(|t(\Fb,\Gbb(n,W))&-t(\Fb,W)| >\eps\bigr) =
\Pr\bigl((t(\Fb,\Gbb(n,W))-t(\Fb,W))^4 >\eps^4\bigr)\\
&\le \frac{\E_x\bigl((t(\Fb,\Gbb(n,W))-t(\Fb,W))^4\bigr)}{\eps^4} =
O\left(\frac1{n^2}\right).
\end{align*}
Hence
\[
\sum_{n=1}^\infty \Pr\left(|t(\Fb,\Gbb(n,W))-t(\Fb,W)| >\eps\right) <\infty,
\]
and so by the Borel-Cantelli Lemma, with probability $1$,
$|t(\Fb,\Gbb(n,W))-t(\Fb,W)|\le\eps$ if $n$ is large enough. With
probability $1$, this holds for every $\eps=1,1/2,1/3,\dots$
simultaneously, which means that
\begin{equation}\label{EQ:FS-DENS}
t(\Fb,\Gbb(n,W))\to t(\Fb,W).
\end{equation}
Since there are only a countable number of $\Fs$-decorated simple
graphs, \eqref{EQ:FS-DENS} holds, with probability $1$, for every
$\Fs$-decorated graph $\Fb$ simultaneously. We can use a similar
argument to show that in addition, with probability $1$, for each positive integer $k$,
$\lim_{n\to\infty}\|W_{n}\|_k=\|W\|_k$. Since for any $\ZZ$-graphon $U$ and any $p<k$ we have $\|U\|_p\leq\|U\|_k+1$, it follows that $\|W_n\|_p$ remains bounded for every $p\ge1$. In this case,
\eqref{EQ:FS-DENS} also holds for every $\BB$-decorated graph by
Lemma \ref{PROP:GENERATE}, and so we get that, with probability $1$,
$\Gbb(n,W)\to W$ as $n\to\infty$.
\end{proof}

\section{Convergence of multigraphs}
\label{section:multi}

The main application of the above theory pertains to sequences of
multigraphs where there is no global bound on edge multiplicities,
and we are interested in convergence of node-and-edge homomorphism
numbers.

\subsection{Multigraphs as Banach decorated graphs}
\label{subsection:multi}

To capture convergence of multigraphs, we would like to use
probability distributions on $\N$ to decorate edges, because this
would allow us to generate a random multiplicity for the edge. We
want finite node-and-edge densities, and therefore we have to require
that these distributions have finite moments. Such distributions
generate the linear space $\JJ$ of signed measures on $\N$ having
finite moments. For $\gamma\in\JJ$, we denote by
$\gamma^{(p)}=\sum_{n=0}^\infty \gamma(n)n^p$ its $p$-th moment
($p\ge0$). We can also think of signed measures as sequences indexed
by $\N$, and then
\[
\JJ=\left\{x\in\R^{\N}:~\sum_n |x_n| n^p<\infty
\text{ for all }p\ge0\right\}.
\]
Let $\JJ_+=\{x\in\JJ:~x_n\ge0\ \forall n\in\N\}$. Note that
$\PP(\N)=\{x\in\JJ_+:~\sum_n x_n=1\}$ consists of those probability
distributions on $\N$ having finite moments.

Based on the special case of multigraphs with bounded edge
multiplicity, the space $\JJ$ (more exactly, its subset $\PP(\N)$)
seems to be the right set to define limit graphons of multigraph
sequences. However, there are two technical difficulties with this.
First, there is no good way to turn $\JJ$ into a Banach space with a
pre-dual, which we would need to be able to apply the machinery
developed above. Second, the moments of distributions on $\N$ do not
determine the distribution (but they determine subgraph
multiplicities). In other words, we cannot (and should not)
distinguish between two distributions with the same moments.

The first problem will be addressed by introducing a weight function;
the second, by taking the factor space $\wh\JJ$ of $\JJ$, in which
two sequences are identified if they have the same moments.

Let us fix a weight function $\rho\in\JJ_+$ (when we speak of a
weight function, we assume that $\rho(n)>0$ for all $n\in\N$). Let
the space
\[
\CC_\rho:=\{f:\N\to\R|f(m)\rho(m)\to 0\}
\]
be equipped with the norm $\|f\|^\rho:=\|f\cdot\rho\|_\infty$. The
dual of this Banach space is
\[
\JJ_\rho:=\left\{x\in\R^\N:
~\left|~\sum_m \frac{|x_m|}{\rho(m)} <\infty\right.\right\}
\]
with the norm
\[
\|x\|_\rho:=\sum_m \frac{|x_m|}{\rho(m)}.
\]
We note that $\JJ_\rho\subset\JJ$.

Since $\rho\in\JJ_+$, we have $(p(0),p(1),...)\in\CC_\rho$ for every
polynomial $p\in R[X]$; with some abuse of notation, we will denote
this sequence by $p$. Let $\RR$ denote the linear space of all such
sequences, and let $\RR_\rho:=\overline{\RR}^{\|\cdot\|^\rho}$. Then
$\RR_\rho$ is a closed subspace of $\CC_\rho$, and so its dual
$\wh{\JJ}_\rho$ is a factor space of $\JJ_\rho$. The elements of
$\wh{\JJ}_\rho$ are equivalence classes of sequences; two of these
sequences are equivalent if and only if they have the same moments (see, e.g., \cite[Thm. III.10.1]{Conway}).

We need a lemma about the existence of weight functions $\rho$. We
call a family $\MM$ of probability distributions on $\N$ {\it
moment-bounded}, if for every $p\ge1$ there is a constant $c_p>0$
such that
\[
\sum_{m\ge 0} \mu(m)m^p <c_p \qquad \text{for all~}\mu\in\MM.
\]
Given a weight function $\rho\in\JJ_+$, we call the family $\MM$ {\it
$\rho$-smooth}, if for every $p\ge1$ there is a constant $c'_p>0$
such that
\[
\sum_{m\ge 0}\frac{\mu(m)}{\rho(m)^p}<c'_p\qquad
\text{for all~}\mu\in\MM.
\]

\begin{lemma}\label{LEM:SMOOTH}
A family $\MM$ of probability distributions is moment-bounded if and
only if there exists a weight function $\rho\in\JJ_+$, such that
$\MM$ is $\rho$-smooth.
\end{lemma}

\begin{proof}
The ``if'' part is easy: Clearly $\rho(m)\le 1/m$ if $m$ is large
enough (say $m\ge M$), and hence for any $p\geq 1$ $m^p\le M^p+1/\rho(m)^p$. Thus
\[
\sum_m \mu(m)m^p \le \sum_m
\mu(m)M^p +\sum_m\frac{\mu(m)}{\rho(m)^p} \le M^p + c'_p.
\]

To prove the ``only if'' direction, note that for every $\mu\in\MM$,
\[
\sum_{m=s}^\infty m^p \mu(m)
\le \frac{1}{s}\sum_{m=s}^\infty m^{p+1} \mu(m)\le \frac{c_{p+1}}{s}.
\]
Hence for $s_j=2^j c_{j^2+1}$, we have
\[
\sum_{m=s_j}^\infty m^{j^2} \mu(m)\le\frac{1}{2^j}.
\]
Let $\rho(m)=m^{-j}$ for $s_j\le m < s_{j+1}$ ($0\leq j$) and $\rho(m)=1$ for $m\leq s_0$. Clearly
$\rho\in\JJ_+$, and
\begin{align*}
\sum_{m=0}^\infty \rho(m)^{-p} \mu(m) &=
\sum_{m=0}^{s_p-1}\rho(m)^{-p} \mu(m)+\sum_{j=p}^\infty
\sum_{m=s_j}^{s_{j+1}-1} \rho(m)^{-p} \mu(m)\\
&\le\sum_{m=0}^{s_p} m^{p^2} \mu(m)+\sum_{j=p}^\infty
\sum_{m=s_j}^{s_{j+1}-1} m^{j^2} \mu(m)\\
&\le\sum_{m=0}^\infty m^{p^2} \mu(m)+\sum_{j=p}^\infty
\sum_{m=s_j}^\infty m^{j^2} \mu(m)\\
&\le c_{p^2} + \sum_{j=p}^\infty \frac{1}{2^j}\le c_{p^2}+1.
\end{align*}
Letting $c'_p=c_{p^2}+1$, the lemma is proved.
\end{proof}

\begin{remark}\label{rem:smooth_norm}
Let $G$ be a multigraph on $n$ vertices such that its multiedge distribution $\mu$ is $\rho$-smooth with constants $c_p'$. Then with $\ZZ=\widehat{\mc{J}}_\rho$ equipped with the $\rho$ norm, we have that the corresponding $\ZZ$-graphon $W_G$ satisfies
\begin{align*}
\|W_G\|_p^p=\frac1{n^2}\sum_{1\leq i\leq j\leq n} \left(\|\delta_{\left\{g_{ij}\right\}}\|^\rho\right)^p=
\frac1{n^2}\sum_{1\leq i\leq j\leq n} \left(\frac{1}{\rho(g_{ij})}\right)^p=\sum_{\ell=0}^\infty \frac{\mu(\ell)}{\rho(\ell)^p}<c_p',
\end{align*}
meaning that the constants in the $\rho$-smoothness translate to the $p$-th power of the constants for the various $L^p$-bounds ($1\leq p<\infty$).\\
In particular a sequence $(G_n)$ of $\rho$-smooth multigraphs will remain uniformly $L^p$ bounded as $\widehat{\mc{J}}_\rho$-decorated graphs $(\Gb_n)$.
\end{remark}

Let $\Fs:=\{X^k:~k\in\N\}\subseteq\RR$. By definition, the set $\Fs$
linearly generates $\RR$, and hence it is generating in $\RR_\rho$, which thus is a separable space.
We may turn a loopless multigraph $\Fb$ into an $\Fs$-decorated graph
$\Fb'$ by decorating each edge $ij$ with multiplicity $m$ by $X^m$.

To each multigraph $\Gb=(G,g)$ (where $G$ is the underlying simple
graph and $g_{uv}$ is the multiplicity of the edge $uv$), we can
associate a $\JJ_\rho$-decorated complete graph by decorating
each edge with multiplicity $m$ by the probability distribution
concentrated on $m$. Through the factor mapping we obtain a
$\wh{\JJ}_\rho$-decorated complete graph $\Gb''$. (Note that these
measures in $\JJ_\rho$ have finite support, and therefore they
are determined by their moments. So the factor mapping distinguishes
the edge-labels in $\Gb''$.) For the node-and-edge homomorphism
density between two multigraphs $\Gb=(G,g)$ and $\Fb=(F,f)$, we have
$t(\Fb,\Gb)=t(\Fb',\Gb'')=t(\Fb',W_\Gb)$ as defined for
$\Fs$-decorated graphs and $\wh{\JJ}_\rho$-graphons. This can be
expressed in terms of the moments as
\begin{align}\label{EQ:T-MOM}
t(\Fb,\Gb)&=\sum_{\ph:V(F)\to V(G)} \prod_{ij\in E(F)}
\langle X^{f_{ij}}, g_{\ph(i)\ph(j)}\rangle\nonumber\\
&= \sum_{\ph:V(F)\to V(G)} \prod_{ij\in E(F)} g_{\ph(i)\ph(j)}^{(f_{ij})}.
\end{align}

Now we turn to convergence of multigraph sequences. Recall that a
sequence of $(\Gb_n)$ of multigraphs is {\it node-and-edge
convergent}, if $t(\Fb,\Gb_n)$ is a convergent sequence for every
multigraph $\Fb$.

\begin{thm}\label{THM:MULTI-CONV}
Let $(\Gb_n)$ be a sequence of multigraphs such that for every
multigraph $\Fb$, the sequence $t(\Fb,\Gb_n)$ (in the node-and-edge
sense) is convergent. Then there exists a weight function $\rho\in \JJ_+$
and a $\wh{\JJ}_\rho$-graphon $W$ such that $t(\Fb,\Gb_n)\to
t(\Fb',W)$.
\end{thm}

\begin{proof}
Let $\mu_n$ denote the edge-multiplicity distribution of $\Gb_n$,
i.e., $\mu_n(m)$ is the probability that uniformly chosen random
nodes $i$ and $j$ of $\Gb_n$ are connected by $m$ edges. (Nodes $i=j$
are connected by $0$ edges.) Let us observe that every 
moment of the probability measures $\mu_n$ is uniformly bounded
(independently of $n$). Indeed, let $\Bb_p=(K_2,p)$ denote the graph
consisting of two nodes connected by $p$ edges (the $p$-bond), then
\[
\sum_m\mu_n(m)m^p =t(\Bb_p,\Gb_n) \le c_p,
\]
since the sequence $(t(\Bb_p,\Gb_n):~n\in\N)$ is convergent and hence
bounded.

Lemma \ref{LEM:SMOOTH} implies that there is a weight function
$\rho\in\JJ_+$ such that the family $\MM=\{\mu_n\}$ is $\rho$-smooth.

In light of Remark \ref{rem:smooth_norm} we may apply
Theorem \ref{thm:Main2} to obtain a $\wh{\JJ}_\rho$-graphon $W$ such that
\[
\lim_{n\to\infty} t(\Fb,\Gb_n)=\lim_{n\to\infty} t(\Fb',\Gb''_n)=t(\Fb',W)
\]
for every multigraph $\Fb$.
\end{proof}

An unpleasant point is that the space $\wh{\JJ}_\rho$ is awkward to
define and work with. One way out would be to ignore that several
distributions may have the same moments, and just select one
distribution as $W(x,y)$ out of the equivalence class. We don't know,
however, whether this can be done in a measurable way.

A better solution is to encode $W(x,y)\in\wh{\JJ}_\rho$ by its moment
sequence. Let us call a sequence $(a_0,a_1,\dots)$ of real numbers a
{\it $\N$-moment sequence}, if there is a probability distribution
$\mu$ on $\N$ such that $a_p=\sum_{m=0}^\infty \mu(m)m^p$. Define
\[
W_p(x,y) = \sum_{m=0}^\infty \pi(m)m^p,
\]
where $\pi\in W(x,y)$ is an arbitrary probability distribution from
the equivalence class. By the definition of equivalence, this is
independent of the choice of $\pi$, and in fact the sequence
$(W_0(x,y),W_1(x,y),\dots)$ uniquely determines the equivalence
class. Then for any multigraph $\Fb=(F,f)$,
\[
t(\Fb',W) = t(\Fb,W_0,W_1,\dots):=
\int\limits_{[0,1]^{V(F)}} \prod_{ij\in E(F)} W_{f(ij)}.
\]
Thus we get the following corollary to Theorem \ref{THM:MULTI-CONV}:

\begin{cor}\label{COR:MULTI-CONV-MOM}
For every node-end-edge convergent sequence $(\Gb_n)$ of multigraphs
there is a sequence $(W_0,W_1,\dots)$ of measurable, symmetric
functions $\LL\ni W_i:~[0,1]^2\to [0,\infty)$ such that for every
$(x,y)\in[0,1]^2$, the sequence $(W_0(x,y),W_1(x,y),\dots)$ is an
$\N$-moment sequence, and $t(\Fb,\Gb_n)\to t(\Fb,W_1,W_2,\dots)$ for
every multigraph $\Fb$.
\end{cor}

\subsection{$W$-random multigraphs}\label{subsec:wrandom2}

We have seen how to sample from a $\wh{\JJ}_\rho$-graphon $W$; the result
is a $\wh{\JJ}_\rho$-decorated graph. However, in the case of simple
graphs, sampling goes one step further, and we get a simple graph
rather than a $[0,1]$-weighted graph. We can make the corresponding
step in the case of multigraphs as well: in this case the limit
objects are $\wh{\JJ}_\rho$-graphons, and the first sampling, as described
in Section \ref{sec:wrandom1}, gives a $\wh{\JJ}_\rho$-decorated graph
$\Gbb(n,W)$. For each edge $ij$ ($i,j\in[n]$, $i\not=j$), we select a
representative $W'(i,j)$ from the equivalence class $W(i,j)$, and
generate a multiplicity $m(i,j)$ from the distribution $W'(i,j)$, to
get a multigraph $\Gbb^\rand(n,W)$. Note the unfortunate
indeterminacy in selecting the probability distribution $W'(i,j)$;
Corollary \ref{cor:WRANDM-CONV} below remains valid even if an
adversary selects these distributions.

To describe this more exactly, let $\Gb=(G,g)$ be a $\JJ_\rho$-decorated
complete graph. By a {\it randomization} of $\Gb$, we mean a (random)
multigraph $\Gb^\rand$, where the multiplicity $G_{uv}$ of an edge
$uv$ is randomly selected from the distribution $g_{uv}$,
independently for different edges. It will be convenient to assume
that $g_{uu}$ is concentrated on $0$, so $G_{uu}=0$.

\begin{lemma}\label{LEM:T-CONT}
Let $\Fb$ be a multigraph with $L$ edges. Then there is a constant
$c=c(\Fb)>0$ such that if $\Gb=(G,g)$ is a $\JJ_\rho$-decorated complete
graph with $n$ nodes and $\eps>0$, then
$|t(\Fb,\Gb)-t(\Fb,\Gb^\rand)|\le\eps $ with probability at least
$1-c\|g^{(4L)}\|_1/(\eps^4n^2)$.
\end{lemma}

\begin{proof}
The proof is similar in structure to the proof of Theorem
\ref{THM:WRAND-CONV}, but the details are different. Let $\Fb=(F,f)$,
where $F$ is a simple graph with $k$ nodes and $l$ edges, $f_{ij}$ is
the multiplicity of edge $ij$, and $L=\sum_{ij}f_{ij}$ is the number
of edges in the multigraph $\Fb$. Let $\Gb=(G,g)$, and let $\Gb$ have
$n$ nodes.  Let $g_{u,v}^{(p)}=\E(G_{u,v}^p)$ denote the $p$-th
moment of $G_{u,v}$. This definition implies that for every $p\ge1$,
\begin{equation}\label{EQ:GPP}
\E(\|\Gb^\rand\|_p^p) = \frac1{n^2}\sum_{u,v\in V(G)} \E(G_{u,v}^p)
= \frac1{n^2}\sum_{u,v\in V(G)} g_{u,v}^{(p)} =\|g^{(p)}\|_1.
\end{equation}

Recall that
\[
t(\Fb,\Gb) = \frac1{n^k}\sum_{\ph:V(F)\to V(G)} \hom_\ph(\Fb,\Gb),
\]
where
\[
\hom_\ph(\Fb,\Gb) = \prod_{ij\in E(F)} g^{(f_{ij})}_{\ph(i),\ph(j)} =
\prod_{ij\in E(F)} \E\bigl(G_{\ph(i),\ph(j)}^{f_{ij}}\bigr).
\]
Setting $X_\ph=\hom_\ph(\Fb,\Gb^\rand)-\hom_\ph(\Fb,\Gb)$ and using a
similar expression for $t(\Fb,\Gb^\rand)$, we have
\[
t(\Fb,\Gb^\rand)-t(\Fb,\Gb) =\frac1{n^k}\sum_{\ph:V(F)\to V(G)} X_\ph.
\]
If $\ph$ is injective, then
$\E\bigl(\hom_\ph(\Fb,\Gb^\rand)\bigr)=\hom_\ph(\Fb,\Gb)$ (the
expectation is taken over the randomization of the edge
multiplicities), and so $\E(X_\ph)=0$. This shows that
$t(\Fb,\Gb^\rand)$ will be close to $t(\Fb,\Gb)$ in expectation.

To get that they are close with high probability, we need to estimate
higher moments. We can write
\begin{equation}\label{EQ:M-MOM}
\bigl(t(\Fb,\Gb^\rand)-t(\Fb,\Gb)\bigr)^4=\frac1{n^{4k}}\Bigl(\sum_{\ph:V(F)\to V(G)}
X_\ph\Bigr)^4 = \frac1{n^{4k}}\sum_{\ph_1,\dots,\ph_4} X_{\ph_1}\dots X_{\ph_4}.
\end{equation}
Note that $X_{\ph_i}$ is independent (as a random variable) from all
$X_{\ph_j}$ for which the range of $\ph_j$ is disjoint from the range
of $\ph_i$. This implies that if we take expectation in
\eqref{EQ:M-MOM}, only those terms remain where for every $i$,
$\ph_i$ is non-injective or its range intersects the range of the
other $\ph_j$. Let us also note that if $\ph$ identifies two adjacent
nodes of $F$, then $X_\ph=0$, so to get a non-zero term, all the
$\ph_i$ must map edges of $F$ onto distinct nodes. Call these terms
{\it bad}. It is easy to see that this implies that the ranges of
$\ph_1,\dots,\ph_4$ cover at most $4k- 2$ nodes of $G$, and so the
number of bad $4$-tuples of maps is bounded by
$\binom{n}{4k-2}(4k-2)^{4k}< (4k)^{4k} n^{4k-2}$.

For a particular map $\ph$, we use the bound
\[
|X_\ph|\le \hom_\ph(\Fb,\Gb)+\hom_\ph(\Fb,\Gb^\rand),
\]
which gives
\begin{align*}
|X_{\ph_1}\dots X_{\ph_4}| &\le
\sum_{S\subseteq[4]} \prod_{s\in S}\hom_{\ph_s}(\Fb,\Gb)
\prod_{s\notin S}\hom_{\ph_s}(\Fb,\Gb^\rand)\\
&=\sum_{S\subseteq[4]} \prod_{s\in S}
\prod_{ij\in E(F)} g_{\ph_s(i),\ph_s(j)}^{(f_{ij})}
\prod_{s\notin S}\prod_{ij\in E(F)} G_{\ph_s(i),\ph_s(j)}^{f_{ij}}.
\end{align*}
We think of every summand on the right side as a product of $4L$
terms of the type $(g_{\ph_s(i),\ph_s(j)}^{(f_{ij})})^{1/f_{ij}}$ and
$G_{\ph(i),\ph(j)}$, and use the inequality between the geometric
mean and the power mean with exponent $4L$. We get
\begin{align*}
|X_{\ph_1}\dots X_{\ph_4}|
&\le \sum_{S\subseteq[4]} \frac1{4L}\Bigl(\sum_{s\in S}
\sum_{ij\in E(F)} f_{ij} \bigl(g_{\ph_s(i),\ph_s(j)}^{(f_{ij})}\bigr)^{4L/f_{ij}}
+ \sum_{s\notin S} \sum_{ij\in E(F)}f_{ij} G_{\ph_s(i),\ph_s(j)}^{4L}\Bigr)\\
&= \frac{2}{L}\Bigl(\sum_{s=1}^{4}
\sum_{ij\in E(F)} f_{ij} \Bigr(\bigl(g_{\ph_s(i),\ph_s(j)}^{(f_{ij})}\bigr)^{4L/f_{ij}}
+ G_{\ph_s(i),\ph_s(j)}^{4L}\Bigr).
\end{align*}
Using the inequality that $(g_{u,v}^{(a)})^b\le g_{u,v}^{(ab)}$ for
$a,b\ge1$, we get
\[
|X_{\ph_1}\dots X_{\ph_4}|\le \frac{2}{L}\sum_{s=1}^{4}
\sum_{ij\in E(F)}f_{ij}\Bigl(g_{\ph_s(i),\ph_s(j)}^{(4L)}
+ G_{\ph_s(i),\ph_s(j)}^{4L}\Bigr).
\]
Note that there are no terms here with $\ph_s(i)=\ph_s(j)$.
Furthermore, if $\ph_1,\dots,\ph_4$ gives a bad term, then so does
$\pi\circ\ph_1, \dots \pi\circ\ph_4$ for any permutation of $V(G)$.
If we average this over all permutations $\pi$ of $V(G)$, we get
every term $g_{u,v}^{(4L)}$ and $G_{u,v}^{4L}$ ($u\not=v$) the same
number of times:
\begin{align}\label{EQ:X-AVE}
\frac1{n!}\sum_\pi |X_{\pi\circ\ph_1}\dots X_{\pi\circ\ph_4}|
&\le \frac{2}{Ln(n-1)} \sum_{u\not=v\in V(G)}
\sum_{ij\in E(F)} f_{ij} \bigl(g_{u,v}^{(4L)}
+ G_{u,v}^{4L}\bigr)\nonumber\\
&= \frac{2}{L}
\sum_{ij\in E(F)} f_{ij} \bigl(\|g^{(4L)}\|_1
+ \|G\|_{4L}^{4L}\bigr)\nonumber\\
&=  2\|g^{(4L)}\|_1
+ 2\|G\|_{4L}^{4L}
\end{align}
Hence summing over bad terms, we get
\[
\sum_{\text{bad}} |X_{\ph_1}\dots X_{\ph_4}|
\le 2 (4k)^{4k} n^{4k-2} \bigl(\|g^{(4L)}\|_1
+ \|G\|_{4L}^{4L}\bigr).
\]
Taking the expectation in \eqref{EQ:M-MOM} and using \eqref{EQ:GPP}
gives
\begin{align*}
\E\bigl((t(\Fb,\Gb^\rand)-&t(\Fb,\Gb))^4\bigr)
= \frac1{n^{4k}}\sum_{\ph_1,\dots,\ph_4~\text{bad}}
\E(X_{\ph_1}\dots X_{\ph_4})\\
&\le \frac{2(4k)^{4k}}{n^2}\bigl(\|g^{(4L)}\|_1+\E(\|G\|_{4L}^{4L})\bigr)
= \frac{4(4k)^{4k}}{n^2}\|g^{(4L)}\|_1.
\end{align*}
Hence
\begin{align*}
\Pr(|t(\Fb,\Gb^\rand)-t(\Fb,\Gb)|>\eps)
&= \Pr((t(\Fb,\Gb^\rand)-t(\Fb,\Gb))^4>\eps^4)\\
&\le \frac{\E((t(\Fb,\Gb^\rand)-t(\Fb,\Gb))^4)}{\eps^4}
\le \frac{4(4k)^{4k}}{n^2\eps^4}\|g^{(4L)}\|_1,
\end{align*}
which proves the lemma.
\end{proof}

Combining this lemma with Theorem \ref{THM:WRAND-CONV}, we get that
every $\wh{\JJ}_\rho$-decorated graphon is the limit of multigraphs:

\begin{cor}\label{cor:WRANDM-CONV}
For every $\wh{\JJ}_\rho$-graphon $W$ and every multigraph $\Fb$,
\[
t(\Fb,\Gbb^\rand(n,W)) \to t(\Fb,W) \qquad(n\to\infty)
\]
with probability $1$.
\end{cor}

\subsection{Node-convergence vs. node-and-edge convergence}\label{SEC:N-NE}

When looking at sequences of multigraphs, several convergence notions
are available. From the combinatorial standpoint, we have convergence
of node-and-edge homomorphism densities and convergence of node
homomorphism densities. From a more abstract point of view, we have
convergence in the compactification sense \cite{LSz8,KoRath,Hombook},
i.e., we take the one-point compactification $\overline{\N}$ of $\N$,
and consider multigraphs as graphs decorated with measures from
$\RR(\overline{\N})$. Alternatively, we have convergence in the weighted
Banach space sense introduced in the previous section. As we have
seen, the latter convergence notion corresponds to the combinatorial
node-and-edge convergence. On the other hand, compactification
convergence corresponds to convergence in the node-homomorphism
sense. To be precise, let $\Gb^{(t)}$ denote the multigraph obtained
from the multigraph $\Gb$ by truncating its edge-multiplicities at
$t$. The following proposition characterizes node-convergence
\cite{LSz8,KoRath,Hombook}:

\begin{prop}\label{PROP:COMPACT}
For a sequence of multigraphs $(\Gb_1,\Gb_2,\dots)$, the following
are equivalent:

\smallskip

{\rm(i)} The sequence is convergent in terms of node-homomorphism
densities.

\smallskip

{\rm(ii)} The sequence is convergent as $\overline{\N}$-labeled
graphs.

\smallskip

{\rm(iii)} For every $t\in\N$, the truncated sequence
$(\Gb_1^{(t)},\Gb_1^{(t)},\dots)$ is convergent as a sequence of
multigraphs with bounded edge-multiplicities.
\end{prop}

Indeed, convergence in the truncated multiplicity sense corresponds to
decorating our test graphs with the truncated polynomials
$q_k(n):=\min\{n^k,k^k\}$. The linear span of the truncated
polynomials is the same as the space generated by the characteristic
functions $f_k:=\one_{\{k\}}$ for all $k$, together with the constant
1 function (which is actually $q_0$). These in turn generate
$C(\overline{\N})$, and thus the three convergence notions are
equivalent.

In the case of graph sequences with uniformly bounded edge
multiplicities, convergence in the node-and-edge homomorphism sense
is also equivalent. How are the two types of convergence related
in the general case? We shall show that these notions are not
equivalent in general, but under appropriate conditions they are
closely related.

We start with two examples.

\begin{example}\label{EXA:THIN}
Let $\Gb_n$ be the multigraph on $[n]$ with $c_n n^2$ edges but with
$o(n^2)$ distinct edges, where $(c_n)$ is an arbitrary bounded
sequence. This graph sequence is convergent in the node-sense. The
truncated graph $\Gb^{(t)}_n$ has $o(n^2)$ edges, and hence it tends
to the identically zero graphon. However, the edge densities
$t(K_2,\Gb_n)=c_n$ do not form a convergent sequence in general, so
this sequence is not convergent in the node-and-edge sense.
\end{example}

\begin{example}\label{EXA:TWO}
Let $\sigma$ and $\tau$ be two different probability distributions on $\N$ with
finite moments and having the same moments (such distributions exist,
see, e.g., \cite{Merx}\footnote
{
The explicit example given there is as follows: given a fixed integer $q\geq 2$, let $\sigma(\{a\})=e^{-q}q^n/n!$ and 
$\tau(\{a\})=e^{-q}q^n\left(\frac{1}{n!}+\frac{(-1)^n}{\prod_{m=1}^n (q^m-1)}\right)$
whenever $a=q^n$ for some $n\in\mb{N}^*$, and let all other natural numbers have zero measure.
}
)
We consider $\sigma$ and $\tau$ as elements of
$\JJ_\rho$ for some appropriate weight function $\rho$.
Let $[\sigma]=[\tau]=:\mu\in\wh{\JJ}_\rho$ denote their equivalence class in the factor space.
 Let $U(x,y)\equiv\sigma$, $W(x,y)\equiv\tau$, and $V(x,y)\equiv\mu$. Then $V$ is a $\wh{\JJ}_\rho$-graphon. 
Next, we wish to apply Corollary \ref{cor:WRANDM-CONV} to the $\wh{\JJ}_\rho$-graphon $V$. But as noted before, we may chose which probability distributions from the given equivalence class we wish to use to generate our random multigraphs.
Therefore we shall generate one family with the help of $U$, and one family with the help of $W$.
Consider the random multigraphs
$\Gb_n=\Gbb^\rand(n,[U])$ and $\Hb_n=\Gbb^\rand(n,[W])$. We know by Corollary \ref{cor:WRANDM-CONV} that $t(\Fb,\Gb_n)\to t(\Fb,V)$ and $t(\Fb,\Hb_n)\to t(\Fb,V)$ almost
surely. 
This means that merging the two
sequences, we get a sequence
$(\Gb_1,\Hb_1,\Gb_2,\Hb_2,\Gb_3,\Hb_3,\dots)$ that is almost surely
convergent in the node-and-edge sense.

On the other hand, there is a $j\in\N$ such that
$\sigma_j\not=\tau_j$, and then the truncated distributions
\[
\sigma^{(j+1)}_i=
  \begin{cases}
    \sigma_i, & \text{if $i\le j$}, \\
    \displaystyle\sum_{i>j}\sigma_i, & \text{if $i\le j+1$}, \\
    0, & \text{otherwise,}
  \end{cases}
\qquad\text{and}\qquad
\tau^{(j+1)}_i=
  \begin{cases}
    \tau_i, & \text{if $i\le j$}, \\
    \displaystyle\sum_{i>j}\tau_i, & \text{if $i\le j+1$}, \\
    0, & \text{otherwise,}
  \end{cases}
\]
are different. The truncated graphs $\Gb_n^{(j+1)}$ can be considered
as random graphs $\Gbb^\rand(n,[U'])$ (where $U'(x,y)\equiv \sigma^{(j+1)}$),
and hence they converge to $[U']$ almost surely, and similarly
$\Hb_n^{(j+1)}\to [W']$ almost surely (where $W'(x,y)\equiv
\tau^{(j+1)}$). But $\sigma^{(j+1)}$ and $\tau^{(j+1)}$ are different
distributions with a finite support, and hence their moments are not
the same; this means that $t(\Bb_p,[U'])\not=t(\Bb_p,[W'])$ for an
appropriate $p\in\N^*$, and hence $t(\Bb_p,\Gb_n^{(j+1)})$ and
$t(\Bb_p,\Hb_n^{(j+1)})$ have different limits. So the truncated
sequence
$(\Gb^{(j+1)}_1,\Hb^{(j+1)}_1,\Gb^{(j+1)}_2,\Hb^{(j+1)}_2,\Gb^{(j+1)}_3,\Hb^{(j+1)}_3,\dots)$
is not convergent, which implies that
$(\Gb_1,\Hb_1,\Gb_2,\Hb_2,\Gb_3,\Hb_3,\dots)$ is almost surely not
convergent in the compactification sense.
\end{example}

\begin{prop}
Suppose the multigraph sequence $\Gb_n$ is convergent in the node
sense. If the numbers $t(\Bb_p,\Gb_n)$ are bounded for each $p\in\N$,
then the graph sequence is also convergent in the node-and-edge
sense.
\end{prop}

\begin{proof}
By Lemma \ref{LEM:SMOOTH}, we can choose a weight function
$\rho\in\JJ_+$ such that the family of edge multiplicity
distributions is $\rho$-smooth. All polynomials and all truncated
polynomials lie in $C(\N,\rho)$. For any polynomial $P$ we have
$\lim_{n\to \infty} P(n)\rho(n)=0$, and hence they all lie in the
closed linear span of the truncated polynomials. Hence the
compactification limit is also a limit in the multigraph sense.
\end{proof}

\subsection{Node-and-edge vs. sample convergence and exchangeability}\label{sect:exch}

For simple graphs and bounded, $[0,1]$-valued graphons, it is well known that homomorphism density convergence is equivalent to sample convergence (cf. \cite[Section 5.2.4]{Hombook}). In our formalism, following Example \ref{EXA:SGRAPH}, this means that given a sequence $W_n:[0,1]^2\to \mb{R}^2$ of symmetric measurable functions with $W_n(x,y)\in\left\{(a,b)\in\mb{R}^2\left|a,b\geq 0, a+b=1\right.\right\}$, we have that $(t(F,W_n))_{n\in\mb{N}^*}$ converges for every simple graph $F$ if and only if the distribution of $\mb{G}^\mk{R}(k,W_n)$ over the set of simple graphs on $k$ vertices converges in the weak-* (or vague) topology -- here equivalent to weak convergence of measures --  as $n\to\infty$ for every $k\in\mb{N}^*$.\\
In more generality, Lov\'asz and Szegedy showed (\cite[Theorem 2.3]{LSz8}) the same equivalence in the setting of compact decorations, i.e., when $\BB=C(\mc{K})$ for some compact separable topological space $\mc{K}$ (see Section \ref{SEC:EXAMPLES}). It is thus a natural question to ask what happens in the case of mulitgraphs with unbounded edge multiplicities.\\
Convergence of the samples would provide a connection to exchangeable arrays and Aldous' representation theorem for these.\\
The first non-trivial hurdle is what topology to choose for the distributions. As $\mb{G}^\mk{R}(k,W_n)$ now lives on a locally compact, rather than a compact space, weak convergence (from the weak duality with $C_b(X)$) and weak-* convergence (from the Riesz-Markov duality $(C_0(X)*=M(X)$). The second is that the former convergence notion would preserve total measure (i.e. a sequence of probability distributions would have a limit that's always a probability distribution), but compactness is lost, whereas the latter allows one to make use of weak-* compactness, but a sequence of probability distributions may actually end up with a zero limit.\\
Thirdly, no matter which we choose, even if we assume that the sequences $\mb{G}^\mk{R}(k,W_n)$ have probability distributions as their limits, the density functions $t(\Fb,\cdot)$ are not bounded (except for zero-decorations), and even less in $C_0$, convergence of distributions does not immediately imply convergence in densities.\\
Issues are also present for the reverse direction, in that unboundedness naturally leads into moment indeterminacy problems. In particular the counter-example from \cite[Section 7.3]{DKK} can be adapted to this setting as well.
\begin{example}\label{ex:counter}
Let $\sigma$ and $\tau$ be two different probability distributions on $\mb{N}$ with
finite moments and having the same moments (as in Example \ref{EXA:TWO}). Denote their n-th moments by $M_n$ $(n\geq0)$.\\
Let further $\{S_i\}_{i\in\mb{N}}$ and $\{T_j\}_{j\in\mb{N}}$ be two partitions of $[0,1]$ into measurable sets such that $\lambda(S_i)=\sigma(\{i\})$ and $\lambda(T_j)=\tau(\{j\})$ for all $i,j\in\mb{N}$. Consider the functions $g_\sigma, g_\tau: [0,1]\rightarrow\mb{R}$ defined by
\begin{eqnarray*}
g_\sigma(x):=n_x & \mbox{ whenever } & x\in S_{n_x},\\
g_\tau(x):=m_x & \mbox{ whenever } & x\in T_{m_x},
\end{eqnarray*}
respectively, and let $W_\sigma, W_\tau: [0,1]^2\rightarrow \widehat{\mc{J}}_\rho$ be defined by $W_\sigma(x,y):=\delta_{\{g_\sigma(x)\cdot g_\sigma(y)\}}$ and $W_\tau(x,y)=\delta_{\{g_\tau(x)\cdot g_\tau(y)\}}$, respectively, where $\rho\in\mc{J}_+$ is chosen so that $W_\sigma,W_\tau$ both have finite $\|\cdot\|_p$-norms for all $1\leq p<\infty$ (such a $\rho$ can be guaranteed with arguments similar to the ones used in the proof of Lemma \ref{LEM:SMOOTH}).\\
Let $\mf{F}$ be a multigraph, let the elements of $V(F)$ be denoted by $v_1,v_2,\ldots,v_k$, let $f_{ij}$ denote the multiplicity of the edge $v_iv_j\in E(F)$, and let $d_i$ denote the degree of vertex $v_i$ (with multiplicities).
It can then easily be seen that we have
\begin{eqnarray*}
t(\mf{F},W_\sigma)&=&\int\limits_{x_1,\ldots,x_k\in[0,1]} \prod_{v_iv_j\in E(F)} (g_\sigma(x_i)\cdot g_\sigma(x_j))^{f_{ij}} dx_1\ldots x_k\\
&=&\prod_{i=1}^k \int_{[0,1]} g_\sigma(x_i)^{d_i} dx_i=\prod_{i=1}^k M_{d_i}.
\end{eqnarray*}
Similar calculations yield $t(\mf{F},W_\sigma)=\prod_{i=1}^k M_{d_i}$, and so the two graphons have the exact same node-and-edge homomorphism densities. Yet, $\mb{G}^\mk{R}(2,W_\sigma)$ is the random multiedge where the multiplicity distribution is the distribution of $g_\sigma(x)g_\sigma(y)$, whilst $\mb{G}^\mk{R}(2,W_\tau)$ is the random multiedge where the multiplicity distribution is the distribution of $g_\tau(x)g_\tau(y)$, and these two are not equal.\\
To turn this into an actual counter-example with respect to the convergence notions at hand, we shall generate a sequence of multigraphs from each of the graphons. Let us first consider $W_\sigma$, the other sequence can be obtained in an analoguous way. For each $n\in\mb{N}$,
let $\sigma_n$ be the probability distribution obtained from $\sigma$ by cutting off it's tail above $n$ and setting it to zero, i.e., $\sigma|_{\ms{P}(\{1,\ldots,n\})}=\sigma_n|_{\ms{P}(\{1,\ldots,n\})}$, $\sigma_n(\{n+1,n+2,\ldots\})=0$, and $\sigma_n(\{0\})=\sigma(\{0\}\cup\{n+1,n+2,\ldots\})$. Next, let the probability distribution $\widehat{\sigma}_n$ be obtained from $\sigma_n$ by choosing some large enough $k_n\in\mb{N}^*$, and letting $\widehat{\sigma}_n(\{m\}):=\lfloor k_n\sigma_n(\{m\})\rfloor/ k_n$ for all $m\geq 1$ in such a way that $\|\widehat{\sigma}_n-\sigma_n\|_{\mr{TV}}<1/n$. Finally, letting $W_{\sigma,n}$ be the graphon obtained from $\widehat{\sigma}_n$ the same way $W_\sigma$ was constructed from $\sigma$, the graphon $W_{\sigma,n}$ will be a stepfunction that represents a multigraph $G_{\sigma,n}$ on $k_n$ vertices. By construction, the total variational distance of $\widehat{\sigma}_n$ and $\sigma$ tends to zero as $n\to\infty$, meaning that $\mb{G}^\mk{R}(k,W_{\sigma,n})$ will tend in total variation distance to $\mb{G}^\mk{R}(k,W_{\sigma})$. Note that convergence also holds in $L^2$, and so a fortiori in jumble norm as well. On the other hand, $\|W_\sigma-W_{\sigma,n}\|_2$ will also tend to zero, and by Corollary \ref{prop:DecDist}, 
$\lim_{n\to\infty} t(\mf{F},W_{\sigma,n})=t(\mf{F},W_\sigma)$. In other words, the multigraph sequences $(G_{\sigma,n})_{n\in\mb{N}^*}$ and $(G_{\tau,n})_{n\in\mb{N}^*}$ have the same limit with respect to node-and-edge densities, but have two different limits with respect to convergence of samples or convergence in jumble norm.
\end{example}

We do expect to have settings (as in the previous section) where these convergence notions coincide, but it is beyond the scope of the present paper to go investigate the details.

\section{Concluding remarks}

\begin{ass}\label{REM:WEIGHT}
An almost identical construction as in Section \ref{section:multi}
can be used to define and study convergent sequences of edge-weighted
graphs with no universal bound on the weights. In this case we use as
a weight function an appropriate function $\rho: [0,\infty) \to\R^+$,
and we replace the set of natural numbers used in the previous
example by the set of nonnegative reals, summation by integral etc.
Caution: one has to distinguish more carefully functions and signed
measures (which were interchangeable in the discrete case above). We
don't go into the details of this.
\end{ass}

\begin{ass}\label{REM:BCCZ}
Using a generalized H\"older Inequality, Borgs, Chayes, Cohn and Zhao
\cite{BCCZ} proved the following inequality stronger than Lemma
\ref{LEM:T-INJ} and inequality \eqref{EQ:HOLD}:
\[
t(F,w) \le  \prod_{ij\in E(F)} \|w_{ij}\|_{\Delta(F)},
\]
where $\Delta(F)$ is the maximum degree in $F$. We could improve
several of our bounds using similar methods. This has not been our
goal in this paper, but it remains an interesting open problem to
extend our results in this direction.
\end{ass}

\begin{ass}\label{}
The theory of simple graph limits  is closely related to the
characterization of homomorphism functions. Such characterizations
are known, among others, for simple graph parameters of the form
$\hom(.,G)$ (where $G$ is a simple graph with loops, or an
edge-weighted graph, or a node-and-edge-weighted graph), and also for
parameters of the form $t(.,W)$, where $W$ is (bounded) graphon.
Extending these characterizations to the Banach space decorated, or
compact decorated, case seems to be a challenging problem.
\end{ass}

\section*{Acknowledgements}
The authors would like to thank S. Janson and the anonymous referee for their thorough read of the original manuscript. Their valuable comments and suggestions have led to a more self-contained final paper.
\bibliographystyle{amsplain}

\begin{thebibliography}{10}

\bibitem{BCCZ}
C.~Borgs, J.T.~Chayes, H.~Cohn and Y.~Zhao: An $L^p$ theory of sparse
graph convergence I: limits, sparse random graph models, and power law
distributions,
\textit{Trans. Amer. Math. Soc.} \textbf{372} (2019), 3019--3062.

\bibitem{BCLSV1} C.~Borgs, J.T.~Chayes, L.~Lov\'asz, V.T.~S\'os
    and K.~Vesztergombi: Convergent Graph Sequences I: Subgraph
    frequencies, metric properties, and testing, {\it Advances in Math.}
    {\bf 219} (2008), 1801--1851.

\bibitem{BCLSV2} C.~Borgs, J.T.~Chayes, L.~Lov\'asz, V.T.~S\'os
    and K.~Vesztergombi: Convergent Graph Sequences II: Multiway Cuts
    and Statistical Physics, {\it Annals of Math.}  {\bf 176}  (2012),
    151--219.

\bibitem{Conway} J.B. Conway, A course in Functional analysis (2nd Ed.), Graduate texts in mathematics \textbf{96}, \textit{Springer} (1990)

\bibitem{FK} A.~Frieze and R.~Kannan: Quick approximation to
    matrices and applications, {\it Combinatorica} {\bf 19} (1999),
    175--220.

\bibitem{Janson} S. Janson: Graphons, cut norm and distance, couplings and rearrangements, NYJM Monographs {\bf 4}, \textit{New York Journal of Mathematics}, 
State University of New York, University at Albany,
Albany, NY, (2013)

\bibitem{KoRath}
I.~Kolossv\'ary and B.~R\'ath: Multigraph limits and exchangeability,
{\it Acta Math. Hung.} {\bf130} (2011), 1--34.

\bibitem{DKK} D. Kunszenti-Kovács: Uniqueness of Banach space valued graphons, \textit{Journal of Mathematical Analysis and Applications} \textbf{474} (2019), 413--440.

\bibitem{Hombook}
L.~Lov\'asz, Large graphs, graph homomorphisms and graph limits,
\textit{AMS} (2012).


\bibitem{LSz1}
L.~Lov\'asz and B.~Szegedy: Limits of dense graph sequences, {\it
J.~Combin.\ Theory B} {\bf 96} (2006), 933--957.

\bibitem{LSz8} L.~Lov\'asz, B.~Szegedy: Limits of compact
decorated graphs,\\
http://arxiv.org/abs/1010.5155

\bibitem{Merx} J.-P. Merx, \\
\url{https://www.mathcounterexamples.net/determinacy-of-random-variables/}

\bibitem{Musial}
K. Musial, Pettis integral, Handbook of Measure Theory I, \emph{North Holland}, (2002)


\bibitem{Tho}
A.~Thomason: Pseudorandom graphs, in: {\it Random graphs '85}
North-Holland Math. Stud. {\bf 144}, North-Holland, Amsterdam, 1987,
307--331.

\end{thebibliography}

\textsc{D\'avid Kunszenti-Kov\'acs.} Alfr\'ed R\'enyi Institute of Mathematics, Budapest, Hungary. \\ \texttt{daku@renyi.hu}

\textsc{L\'aszl\'o Lov\'asz.} Alfr\'ed R\'enyi Institute of Mathematics and Institute of Mathematics, Eötvös Lor\'and University, Budapest, Hungary. \\ \texttt{laszlo.lovasz@ttk.elte.hu}

\textsc{Bal\'azs Szegedy.} Alfr\'ed R\'enyi Institute of Mathematics, Budapest, Hungary. \\ \texttt{szegedy.balazs@renyi.hu}

\end{document}